\documentclass[12pt]{amsart}

\usepackage{upref,amssymb,bm,mathrsfs}
\usepackage{enumerate}

\usepackage[centering,width=6in,truedimen]{geometry}
\usepackage{mathtools}
\usepackage{cite}
\usepackage{verbatim}

\usepackage[colorlinks,pagebackref,pdfstartview=FitH]{hyperref}

\theoremstyle{plain}
\newtheorem{thm}{Theorem}[section]
\newtheorem{prop}[thm]{Proposition}
\newtheorem{lem}[thm]{Lemma}
\newtheorem{cor}[thm]{Corollary}

\newtheorem{ques}[thm]{Question}
\newtheorem*{ques*}{Question}

\theoremstyle{definition}

\newtheorem{exmp}[thm]{Example}

\theoremstyle{remark}
\newtheorem{rem}[thm]{Remark}

\DeclareMathOperator{\dist}{\mbox{dist}}

\newcommand{\N}{\mathbb N}

\newcommand{\R}{\mathbb R}
\newcommand{\RR}{\mathbb R}
\newcommand{\C}{\mathbb C}
\newcommand{\Z}{\mathbb Z}
\newcommand{\supp}{\mbox{supp}}

\def\a{\mathbf{a}}
\newcommand{\A}{\mathcal{A}}
\numberwithin{equation}{section}
\newcommand{\bx}{\mathbf x}
\def \bxi {{\boldsymbol{\xi}}}

\renewcommand{\epsilon}{\varepsilon}

\title{Self-similar and self-conformal measures with slow Fourier decay}

\author{Simon Baker}

\address{Simon Baker\\
Department of Mathematical Sciences\\
Loughborough University\\
Loughborough\\
LE11 3TU, UK
}
\email{simonbaker412@gmail.com}

\author{Amlan Banaji}

\address{
Amlan Banaji\\
Department of Mathematics and Statistics\\
University of Jyväskylä\\
Jyväskylä, FI-40014\\
Finland
}
\email{banajimath@gmail.com}

\subjclass{42A38 (Primary), 28A80, 11J71 (Secondary)}

\keywords{Fourier decay, Rajchman measures, self-similar measures, self-conformal measures}

\begin{document}

\begin{abstract}
	Given any function $\phi \colon [0,\infty)\to (0,1]$ satisfying $\lim_{\xi\to\infty}\phi(\xi) = 0$, we prove the existence of i) self-similar measures and ii) nonlinear $C^{\infty}$ self-conformal measures which are Rajchman and whose Fourier transform $\widehat{\mu}$ satisfies 
    \[ \limsup_{\xi\to\infty}\frac{|\widehat{\mu}(\xi)|}{\phi(\xi)}>0.\] 
    Moreover, we derive new sufficient conditions for a self-conformal measure to be Rajchman, and construct an explicit self-similar measure $\mu$ such that $\mu$ almost every $x$ is normal in base $10$ but the sequence $(10^{n}x \mod 1)_{n=1}^{\infty}$ equidistributes extremely slowly.
\end{abstract}

\maketitle

\section{Introduction}

\subsection{Overview}
Given a compactly supported Borel probability measure $\mu$ on $\R$, its Fourier transform $\widehat{\mu} \colon \R \to \C$ is defined by $\widehat{\mu}(\xi) = \int_{\R} e^{2 \pi i \xi x} d\mu(x)$. A measure $\mu$ is said to be Rajchman if $\lim_{|\xi|\to\infty}\widehat{\mu}(\xi)=0$. If $\mu$ is Rajchman, it is natural to ask whether $\widehat{\mu}$ converges to zero at a given rate. This is an important problem with applications to many different areas of mathematics. In combinatorics it is known that if a closed set has sufficiently large Hausdorff dimension and supports a measure whose Fourier transform decays to zero sufficiently quickly, then this implies the existence of patterns within the set (e.g. three term arithmetic progressions)~\cite{FGP,LP}. 
In number theory, it is known that suitable upper bounds for the Fourier transform of a measure $\mu$ can be used to establish properties of $\mu$ typical points related to Diophantine approximation and normality~\cite{PVZZ}. 
In particular, a well known result due to Davenport, Erd\H{o}s and LeVeque~\cite{DEL} implies that if $\widehat{\mu}(\xi) = O((\log \log |\xi|)^{-1-\varepsilon})$ for some $\varepsilon>0$, then $\mu$ almost every $x$ is normal in all bases. 
In quantum chaos, knowledge about the Fourier transform of a measure can be used to establish fractal uncertainty principles for the support of the measure~\cite{Dyatlov}. 
This is significant, as fractal uncertainty principles can be used to address a variety of problems, such as bounding the $L^2$ mass of eigenfunctions of the Laplacian, and proving observability theorems for PDEs~\cite{DJ,DJN}. 

In this paper, we are interested in finding self-similar and self-conformal measures (defined below) which are Rajchman, but have very slow Fourier decay. This problem has previously been studied in the context of general absolutely continuous measures. The Riemann--Lebesgue lemma states that every absolutely continuous measure is Rajchman. However, it is known that absolutely continuous measures can have arbitrarily slow Fourier decay along a sequence of frequencies, see for instance \cite[Theorem~3]{BT16}. 

For us, an iterated function system (IFS) will be a finite set of uniform contractions $\{S_a \colon [0,1] \to [0,1]\}_{a \in \mathcal{A}}$. 
Given an IFS and a probability vector $(p_{a})_{a\in \mathcal{A}}$, there is a unique Borel probability measure $\mu$ satisfying $\mu = \sum_{a \in \mathcal{A}} p_a S_a \mu$, where $S_a \mu$ is the pushforward of $\mu$ by $S_a$ ($S_{a}\mu(A) = \mu(S_{a}^{-1}(A))$ for $A\subseteq \R$ a Borel set)~\cite{Hutchinson}. 
If all of the $S_a$ are $C^1$ contractions with non-vanishing derivative, then we call $\mu$ a self-conformal measure, and if all of the $S_a$ are similarities then $\mu$ is called a self-similar measure. A self-similar measure is called homogeneous if all contraction ratios are equal. 
We emphasise that throughout this paper we make the standing assumption that all entries in our probability vectors are strictly positive and all self-similar and self-conformal measures are non-atomic. This latter property is equivalent to the contractions in the IFS not sharing a common fixed point. In recent years, the Fourier transform of self-similar and nonlinear self-conformal measures has been a topic of intensive study \cite{ACWW25,AHW,AHWplane,BB25,BKS,BS,BY25,BourgainDyatlovFourier,Bre,JordanSahlsten,LPS,LiSahlsten,SahlstenStevens,SahlstenSurvey,Str}. The main results of this paper, namely Theorems~\ref{t:slowselfsim} and~\ref{t:slowconformal}, show that in these two contexts it is possible for the measures to be Rajchman but have arbitrarily slow Fourier decay. Before giving formal statements, we survey some relevant results. 

\subsection{Self-similar measures}
We begin by remarking that there are self-similar measures that are not Rajchman. For example, every self-similar measure supported on the middle-third Cantor set is non-Rajchman. 
In~\cite{Bre}, Br\'{e}mont classified when a self-similar measure is Rajchman. Results of Shmerkin and Solomyak \cite{ShmSol}, and Solomyak \cite{Sol} demonstrate that for many natural parameterised families of self-similar measures, a typical element will have polynomial Fourier decay, i.e. there exist $C,\eta>0$ such that $|\widehat{\mu}(\xi)|\leq C|\xi|^{-\eta}$ for all $\xi\neq 0$. Despite these results, there are relatively few explicit examples of self-similar measures with polynomial Fourier decay \cite{DFW,Str}. Much more is known when we weaken the decay rate. Li and Sahlsten \cite{LiSahlsten} have shown that under a mild Diophantine assumption on the contraction ratios appearing in the IFS, the Fourier transform of a self-similar measure decays at least at a logarithmic rate, i.e. there exist $C,\eta>0$ such that $|\widehat{\mu}(\xi)|\leq C(\log |\xi|)^{-\eta}$ for all $\xi\neq 0$. Varj\'u and Yu~\cite[Corollary~1.6]{VarjuYu} have given further conditions for a self-similar measure to have at least logarithmic Fourier decay.

Very little is known about Rajchman self-similar measures whose Fourier transform decays to zero slowly. To the best of the authors' knowledge, the first result in this direction was due to Kahane~\cite{Kah}. He showed that Bernoulli convolutions parameterised by reciprocals of Salem numbers do not have polynomial Fourier decay. We refer the reader to the survey of Peres, Schlag, and Solomyak~\cite{PeScSo} for a proof of this result. We emphasise that this implies the existence of Rajchman self-similar measures that do not have polynomial Fourier decay. 
In a recent paper~\cite[Appendix~B]{MaSol}, Marshall-Maldonado and Solomyak established an upper bound for the Fourier transform of Bernoulli convolutions parameterised by reciprocals of Salem numbers; this upper bound decays to zero very slowly, in particular slower than any $k$-fold composition of logarithms. They also improved upon Kahane's result and showed that if a Bernoulli convolution is parameterised by the reciprocal of a Salem number $\alpha$ and there exist $\gamma,\beta>0$ such that the Fourier transform satisfies 
\[ \widehat{\mu}(\xi) = O\left((\log_{\alpha}|\xi|)^{-\gamma(\log_{\alpha} \log_{\alpha} \xi)^{\beta}} \right)\] 
then $\beta\leq 1$. 
They also considered biased Bernoulli convolutions, and showed that if such measures have logarithmic Fourier decay then the exponent can be bounded from above by a constant depending upon the Salem number and underlying probability vector.
In work in progress, Paukkonen, Sahlsten and Streck~\cite{PSS} have shown that Li and Sahlsten's result is sharp in the sense that there exist self-similar measures satisfying the Diophantine assumption from~\cite{LiSahlsten}, and therefore possessing logarithmic Fourier decay, whose Fourier transform decays at precisely a logarithmic rate along a subsequence of frequencies. 
With the above results in mind, it is natural to ask whether the Diophantine assumption in~\cite{LiSahlsten} is really needed to guarantee logarithmic Fourier decay. 
\begin{ques}\label{q:selfsim}
    Is it the case that for every Rajchman self-similar measure $\mu$ on $\R$, there exist $\eta,C > 0$ such that 
    \[ |\widehat{\mu}(\xi)| \leq C (\log |\xi|)^{-\eta} \quad \mbox{ for all } |\xi| > 1? \] 
\end{ques}
Our first result shows that the answer to Question~\ref{q:selfsim} is negative in a strong sense. 
\begin{thm}\label{t:slowselfsim}
For every function $\phi \colon [0,\infty) \to (0,1]$ such that $\phi(\xi) \to 0$ as $\xi \to \infty$, there exists a Rajchman homogeneous self-similar measure $\mu$ on $\mathbb{R}$ such that
\begin{equation}\label{e:selfsimslowdecay} 
\limsup_{\xi \to \infty} \frac{|\widehat{\mu}(\xi)|}{\phi(\xi)} >0. 
\end{equation}
\end{thm}

The self-similar measures appearing in the proof of Theorem~\ref{t:slowselfsim} are generated by the IFS 
\[ \left\{S_{0}(x) = \frac{x}{10},\, S_{1}(x) = \frac{x+1}{10},\, S_{2}(x) = \frac{x+t}{10}\right\} \] 
where $t\in [0,1]$ and the probability vector $(1/3,1/3,1/3)$ (call this self-similar measure $\mu_t$). For this family of measures we will in fact prove something stronger than Theorem~\ref{t:slowselfsim}. Given any function $\phi \colon [0,\infty)\to (0,1]$ satisfying $\phi(\xi)\to 0$ as $\xi\to\infty$, we will show that for a topologically generic choice of $t$, the measure $\mu_{t}$ is Rajchman and satisfies~\eqref{e:selfsimslowdecay} (Proposition~\ref{p:fullselfsim}). This is in contrast to a result of Shmerkin and Solomyak~\cite[Proposition~3.1]{ShmSol} which implies that for all $t\in [0,1]$ outside of a set of Hausdorff dimension zero, the measure $\mu_{t}$ has polynomial Fourier decay. 

In the proof of Theorem~\ref{t:slowselfsim}, the values of $t$ for which we show that $\mu_t$ satisfies~\eqref{e:selfsimslowdecay} will be irrational, but extremely well approximated by rationals $t'$ for which $\mu_{t'}$ is Rajchman. The Fourier transform of $\mu_t$ will therefore be large at certain frequencies where $\widehat{\mu_{t'}}$ is large, but a result of Br\'emont~\cite{Bre} and the irrationality of $t$ will imply that $\mu_t$ is Rajchman.
Our method of proof also gives explicit values of $t$ for which $\mu_t$ has very slow Fourier decay, as in Theorem~\ref{e:explicitt} below. 
Recall Knuth's up-arrow notation for large integers: $10 \uparrow \uparrow n$ can be defined recursively by 
\begin{equation}\label{e:knuth}
    10 \uparrow \uparrow  1 = 10 \quad \mbox{ and } \quad 10 \uparrow \uparrow n = 10^{10 \uparrow \uparrow (n-1)} \quad \mbox{ for integers } n > 1. 
\end{equation}
\begin{thm}\label{t:explicitselfsim}
    Letting $t$ be the explicit Liouville number 
    \begin{equation}\label{e:explicitt}
        t = \sum_{n=1}^{\infty} \frac{1}{10 \uparrow \uparrow (3n)}, 
    \end{equation}
    the measure $\mu_t$ is Rajchman but 
    \[ \liminf_{n \to \infty} ((\log \log (10 \uparrow \uparrow (3n+2))) |\widehat{\mu_t}(10 \uparrow \uparrow (3n+2))|) > 0. \]
\end{thm}

\begin{rem}\label{r:moreexplicitexamples}
Theorem~\ref{t:explicitselfsim} just gives one particular example. 
If instead of~\eqref{e:explicitt} we set 
\[ t = \sum_{n=1}^{\infty} \frac{1}{10 \uparrow \uparrow (kn)} \] 
for some integer $k > 3$, then adapting the proof of Theorem~\ref{t:explicitselfsim}, one could show that the Fourier transform of the Rajchman measure $\mu_t$ would decay at least as slowly as 
\[ (\underbrace{\log \circ \dotsb \circ \log}_{k-1 \mbox{ times}} (\xi))^{-1} \] 
along a sequence of integer frequencies $\xi$ which tend rapidly to $\infty$. 
Replacing $\uparrow \uparrow$ by $\underbrace{\uparrow \dotsb \uparrow}_{m \mbox{ times}}$ for some integer $m > 2$ would result in even slower decay. 
\end{rem}

The self-similar measures which we construct to have slow Fourier decay also satisfy some other interesting properties. One such property relates to points in the support of the measure having a slow rate of equidistribution under certain dynamical systems. 
We have already mentioned the result of Davenport, Erd\H{o}s and LeVeque~\cite{DEL} which implies that if a measure $\mu$ has fast enough Fourier decay then $\mu$ almost every $x$ is normal with respect to all bases, i.e. for all integers $b>1$, for $\mu$ almost every $x$ and all intervals $I \subset [0,1]$ we have 
\begin{equation}\label{e:equidistributiondefn}
\lim_{N \to \infty} \frac{1}{N} \# \{1 \leq n \leq N : \{ b^n x \} \in I \} = \mbox{length}(I), 
\end{equation}
where $\{b^n x\}$ denotes the fractional part of $b^n x$. 
Algom, Rodriguez~Hertz and Wang~\cite{AHWequidistribution} have proved that if $\mu$ is a Rajchman self-similar measure, then $\mu$ almost every $x$ is normal for all integers $b>1$. The significance of this result is that it does not assume a rate of Fourier decay. 
A self-contained proof of this result is given by Algom~\cite{AlgomNormality}. 

It is natural to ask for more quantitative equidistribution results. One might hope that under certain conditions one could obtain non-trivial bounds for the left-hand side of~\eqref{e:equidistributiondefn} even if the length of the interval $I$ is allowed to depend on $n$ and shrink to $0$ in a non-summable way. 
Indeed, Theorem~\ref{thm:pvzz} below for arbitrary (i.e. not necessarily self-similar) measures with fast enough Fourier decay is a special case of a result by Pollington, Velani, Zafeiropoulos and Zorin~\cite{PVZZ}. 
To state this theorem we need the following notation. 
Given a sequence of positive numbers $\psi(1),\psi(2),\dotsc$ we let $\Psi(N) \coloneqq \sum_{n=1}^N \psi(n)$ for each positive integer $N$. 
Given also $\gamma \in [0,1)$, for each $x \in \R$ and $N \in \N$ let 
\begin{equation}\label{e:definepvzzcounts}
R(x,N;\gamma,\psi,b) \coloneqq \# \{ 1 \leq n \leq N : \dist(b^{n-1} x - \gamma , \Z) \leq \psi(n) \}. 
\end{equation}
\begin{thm}[\cite{PVZZ}, Theorem~3]\label{thm:pvzz}
    Let $\mu$ be a non-atomic Borel probability measure supported inside $[0,1]$. Assume that there exists $\eta > 2$ such that $\mu(\xi) = O((\log |\xi|)^{-\eta})$ as $|\xi| \to \infty$. Let $\gamma \in [0,1)$ and $\psi \colon \N \to (0,1]$. Then for all $\varepsilon > 0$ and $\mu$ almost every $x$, 
    \begin{equation}\label{e:pvzzconclusion}
    R(x,N;\gamma,\psi,b) = 2 \Psi(N) + O((\Psi(N))^{1/2} (\log (\Psi(N) + 2))^{2+\varepsilon}) 
    \end{equation}
    as $N \to \infty$. 
\end{thm}

The results of Algom~et~al. and Pollington~et~al. together raise the following natural question. 
\begin{ques}
    Without assuming a rate of Fourier decay, if $\mu$ is a Rajchman self-similar supported inside $[0,1]$, does~\eqref{e:pvzzconclusion} necessarily hold for $\mu$ almost every $x$? 
\end{ques}

The following result shows that the answer is negative in a strong sense. 

\begin{thm}\label{thm:slowequi}
    There exists a monotonically decreasing sequence of positive numbers $\psi(1),\psi(2),\dotsc$, and $\gamma \in (0,1)$, such that the following hold: 
    \begin{enumerate}
        \item\label{i:psigrowsquickly} $\limsup_{N \to \infty} \frac{\log \Psi(N)}{\log N} = 1$, so in particular $\lim_{N \to \infty} \Psi(N) = \infty$. 
        \item Letting $X_t$ be the support of the self-similar measure $\mu_t$ from Theorem~\ref{t:explicitselfsim}, for all $x \in X_t$ and $N \in \N$ we have $R(x,N;\gamma,\psi,10) = 0$. 
    \end{enumerate}
\end{thm}

The idea of the proof is to approximate $t$ by the rational number $t_k \coloneqq \sum_{n=1}^{k} \frac{1}{10 \uparrow \uparrow (3n)}$ and use the fact that most points in $X_{t_k}$ and $X_t$ with the same coding have the same decimal expansion for a long initial segment depending upon $k$. 
Crucially, normality in base $10$ fails for points in $X_{t_k}$. The size of the interval which witnesses this failure of normality decreases to $0$ in way that depends on $k$, but this does not prevent property~\eqref{i:psigrowsquickly} from holding.

\subsection{Nonlinear self-conformal measures}
Beyond the self-similar setting, there has been much work proving (qualitative and quantitative) Fourier decay bounds for nonlinear self-conformal measures supported on $\R$ \cite{ACWW25,AHW,BB25,BS,BourgainDyatlovFourier,BYqualitative,JordanSahlsten,Kau2,LPS,QR,SahlstenStevens}. These results demonstrate that many families of nonlinear self-conformal measures have polynomial Fourier decay. It is natural to ask whether any nonlinearity implies polynomial Fourier decay. This is expressed more formally in the following guiding question posed by Sahlsten~\cite{SahlstenSurvey}. 

\begin{ques}[Problem~3.8 from \cite{SahlstenSurvey}]
\label{q:Sahlsten question}
Suppose $\{S_{a}\}_{a\in \A}$ is a $C^{1+\alpha}$ IFS such that one of the maps $S_{a}$ is not a similitude. Does the Fourier transform of every self-conformal measure associated
to $\{S_{a}\}_{a\in \A}$ decay polynomially? 
\end{ques}
We will see that the answer to this question is negative in general. Nonetheless, much progress relating to this question has been made in recent years. It is now known that the answer to Question~\ref{q:Sahlsten question} is yes when each contraction is real analytic. This follows from the results in~\cite{ACWW25,AHW,BB25,BS}. In the context of $C^{2}$ IFSs, if a self-conformal measure does not have polynomial Fourier decay then it must be $C^{2}$ conjugated to a $C^2$ linear IFS. 
By this we mean that there is some $C^{2}$ diffeomorphism $h \colon [0,1] \to [0,1]$ and $C^2$ IFS $\mathcal{T} = \{T_a\}_{a\in \A}$ such that for all $a\in \A$ and $x \in [0,1]$ we have $T_a(h(x)) = h(S_a(x))$, and such that for all $y$ in the attractor of $\mathcal{T}$ and $a\in \A$ we have $T_a''(y) = 0$. 
This is a direct consequence of \cite[Theorem~1.1]{AHW} (see also the related result \cite[Theorem~1.4]{BS}). This result motivates us to study self-conformal measures that can be realised as pushforward measures of self-conformal measures arising from linear IFSs. 
In this paper we will be interested in the special case of self-conformal measures that can be realised as pushforwards of self-similar measures. 
We emphasise that under very mild assumptions on the pushforward map, the pushforward of a self-similar measure can indeed be realised as a self-conformal measure (see Lemma~\ref{lem:pushforwardconformal}). 
    
The following example shows that the answer to Question~\ref{q:Sahlsten question} is negative for a simple reason, even in the family of $C^{\infty}$ IFSs, and that a further refinement is required. 

\begin{exmp}
\label{Example:self-similar image}
Let $\nu$ be the Cantor--Lebesgue measure on $[0,1]$ (so $T_{1}(x)=x/3$ and $T_{2}(x) = (x+2)/3$ with the $(1/2,1/2)$ probability vector). Let $\mu=f\nu$ where $f \colon [0,1]\to [0,1]$ is an increasing $C^{\infty}$ diffeomorphism which is affine on $[2/9,1/3]$ and such that 
\begin{equation}\label{e:wheresecondderivdoesntvanish1} 
    f''(x) \neq 0 \qquad \mbox{ for all } x \in [0,2/9) \cup (3/9,1]. 
    \end{equation} We claim that it is possible to choose $f$ so that the following properties hold:
    \begin{enumerate}
        \item $\mu$ is the self-conformal measure for the $C^{\infty}$ IFS 
        \[ \{S_{a} \coloneqq f\circ T_{a}\circ f^{-1} \colon [0,1]\to [0,1]\}_{a=1}^{2} \] 
        and the probability vector $(1/2,1/2)$. 
        \item For $a\in \{1,2\}$ we have $S_{a}''(x)\neq 0$ apart from at at most finitely many $x\in [0,1]$.
    \end{enumerate}
We will explain how we can choose $f$ so that these properties are satisfied in Section~\ref{ss:Examples} after we have established some useful preliminary machinery. 
It can be shown using the results of~\cite{BB25,ACWW25} that~\eqref{e:wheresecondderivdoesntvanish1} implies that the restriction of $\mu$ to $f([0,1/9])\cup f([2/3,1])$ has polynomial Fourier decay. However, the restriction of $\mu$ to $f([2/9,1/3])$ is not a Rajchman measure. This is because $f$ is an affine map when restricted to $[2/9,1/3]$ and $\nu|_{[2/9,1/3]}$ is not a Rajchman measure, which follows from $\nu$ not being a Rajchman measure and self-similarity. 
It follows that $\mu$ is not a Rajchman measure. 
This example more generally shows how it is possible for a piece of a ``nonlinear'' self-conformal measure to be given by an affine image of a self-similar measure. Indeed it is possible to exploit this observation together with Theorem~\ref{t:slowselfsim} to show that there exist ``nonlinear'' self-conformal measures for $C^{\infty}$ IFSs that are Rajchman but have arbitrarily slow Fourier decay.

The issue with this example is that the way we measure nonlinearity, i.e. that for all $a\in \{1,2\}$ we have $S_{a}''(x)\neq 0$ apart from at at most finitely many $x\in [0,1]$ (or indeed the weaker hypothesis from Question~\ref{q:Sahlsten question} of not being a similitude), is misleading and hides the obvious linearity that is present. 
The important point to realise is that $\mu$ is also a self-conformal measure for the IFS 
\[ \{ S_{(a_{1},a_{2})} \coloneqq f\circ T_{a_1}\circ T_{a_{2}}\circ f^{-1} \colon [0,1]\to [0,1]\}_{1\leq a_{1},a_{2}\leq 2} \] 
and the probability vector $(1/4,1/4,1/4,1/4)$. However, for this IFS it can be shown that $S_{(1,2)}''(x)=0$ for all $x\in [f(2/9),f(1/3)]$.
\end{exmp}
In light of Example~\ref{Example:self-similar image}, the following refinement of Question~\ref{q:Sahlsten question} is natural.

\begin{ques}
\label{Q:refinedquestion}
Suppose $\{S_{a}\}_{a\in \A}$ is a $C^{1+\alpha}$ IFS on $[0,1]$ such that for every $k\in \N$ and $(a_{1},\dotsc, a_{k})\in \A^{k}$ we have $(S_{a_{1}}\circ \cdots \circ S_{a_{k}})''(x)\neq 0$ apart from at at most finitely many $x\in [0,1]$. 
Does it follow that the Fourier transform of every self-conformal measure associated to $\{S_{a}\}$ decays polynomially?  
\end{ques}

Although we will see that the answer to Question~\ref{Q:refinedquestion} is still negative, the validity of the question is supported by the following theorem and its corollary.

\begin{thm}\label{t:pushisrajchman}
    Let $\mu$ be a self-similar measure on $[0,1]$ for an IFS $\{T_{a}\}_{a\in \A}$. Let $f\colon [0,1] \to \RR$ be a $C^2$ map, and let 
    \[ Z \coloneqq \{ x \in [0,1] : f''(x) = 0 \}. \] 
    Then the following statements hold: 
    \begin{enumerate}
        \item\label{i:pushrajchman} Suppose that $\mu(Z)=0$ (note that this is satisfied in particular if $Z$ is finite). Then the image measure $f\mu$ is Rajchman.
        \item\label{i:measureofzeroset} Suppose that $\mu(Z)>0$, that $f'(x) \neq 0$ for all $x \in [0,1]$, and that $f\mu$ is a self-conformal measure for the IFS $\{S_{a} \coloneqq f\circ T_{a}\circ f^{-1}\}_{a\in \A}$.\footnote{Due to the assumption $f'(x) \neq 0$ for all $x \in [0,1]$ and Lemma~\ref{lem:pushforwardconformal}, it is always the case that we can iterate our IFS such that the third condition is satisfied.}
        Then there exists a word $(a_{1},\dotsc,a_{k})\in \cup_{n=1}^{\infty}\A^{n}$ such that
        \[ f\mu\left( \{ x:(S_{a_{1}}\circ \dotsb \circ S_{a_{k}})''(x) = 0 \} \right) > 0. \]
    \end{enumerate}
    \end{thm}
\begin{cor}
\label{cor:rajchman}
Suppose $\{S_{a} \colon [0,1] \to [0,1] \}_{a\in \A}$ is a $C^{2}$ IFS such that for every $(a_{1},\dotsc,a_{k})\in \cup_{n=1}^{\infty}\A^{n}$ we have $(S_{a_{1}}\circ \cdots \circ S_{a_{k}})''(x)\neq 0$ apart from at at most finitely many $x\in [0,1]$. Then every self-conformal measure associated to $\{S_{a}\}_{a\in \A}$ is Rajchman if either of the following properties is satisfied:
\begin{enumerate}
\item $\{S_{a}\}_{a\in \A}$ cannot be conjugated to a $C^{2}$ linear IFS by any $C^2$ diffeomorphism. 
\item $\{S_{a}\}_{a\in \A}$ can be conjugated to a self-similar IFS by a $C^2$ diffeomorphism. 
\end{enumerate}
\end{cor}
\begin{proof}
    If $\{S_{a}\}_{a\in \A}$ cannot be conjugated to a $C^{2}$ linear IFS then every self-conformal measure has polynomial Fourier decay by the aforementioned results of~\cite{AHW,BS} and is therefore Rajchman.\footnote{In this case we do not in fact need our assumption on $(S_{a_{1}}\circ \dotsb \circ S_{a_{k}})''$.}
    
    Let us now suppose we are in the latter case and there exists a self-similar IFS $\{T_{a}\}_{a\in \A}$ and a $C^{2}$ diffeomorphism $f$ such that $S_{a} = f\circ T_{a}\circ f^{-1}$ for all $a\in \A$. 
    Since $f$ is a diffeomorphism it automatically has non-vanishing derivative by the inverse function theorem. 
    Let $\nu$ be a self-conformal measure for $\{S_{a}\}_{a\in \A}$, so directly from the invariance relation $\nu = \sum_{a \in \mathcal{A}} p_a S_a \nu$ we see that $\nu = f\mu$ for some self-similar measure $\mu$ for $\{T_{a}\}_{a\in \A}$. 
    It follows from our second derivative assumption and the fact that $\nu$ is non-atomic that 
    \[ \nu(\{ x:(S_{a_{1}}\circ \cdots \circ S_{a_{k}})''(x)=0 \} ) = 0 \qquad \mbox{ for all } (a_{1},\dotsc,a_{k})\in \cup_{n=1}^{\infty}\A^{n}. \] 
    Thus, by Theorem~\ref{t:pushisrajchman}~\eqref{i:measureofzeroset} together with the fact $\nu = f\mu$, we must have that $\mu(\{x:f''(x)=0\}) = 0$. 
    Now applying Theorem~\ref{t:pushisrajchman}~\eqref{i:pushrajchman}, we see that $\nu$ is Rajchman. As $\nu$ was an arbitrary self-conformal measure for $\{S_{a}\}_{a\in \A}$, this completes our proof. 
\end{proof}

If the second case appearing in Corollary~\ref{cor:rajchman} could be extended to cover general linear IFSs instead of just self-similar IFSs, then the nonlinearity condition appearing in this corollary would provide a nice sufficient condition for the Rajchman property for self-conformal measures coming from $C^{2}$ IFSs. 
We emphasise that there are examples of linear $C^2$ IFSs that are not conjugate to a self-similar IFS for any $C^2$ diffeomorphism (see~\cite{AOHS}), and that our proof of Theorem~\ref{t:pushisrajchman}~\eqref{i:pushrajchman} uses a result from~\cite{ACWW25,BB25} which is only stated for pushforwards of self-similar measures. 
Corollary~\ref{cor:rajchman} also demonstrates than if an IFS is conjugate to self-similar and provides a negative answer to Question~\ref{Q:refinedquestion}, then it must exhibit some subtle behaviour. Namely the mechanisms within the IFS that guarantee the Rajchman property cannot be so strong as to guarantee polynomial Fourier decay.  

Our next result shows that despite Corollary~\ref{cor:rajchman} and the evidence it provides suggesting an affirmative answer to Question~\ref{Q:refinedquestion}, the subtle behaviour mentioned above is in fact possible and the answer to Question~\ref{Q:refinedquestion} is negative in a strong sense even for $C^{\infty}$ IFSs. 

\begin{thm}\label{t:slowconformal}
For every function $\phi \colon [0,\infty) \to (0,1]$ such that $\phi(\xi) \to 0$ as $\xi \to \infty$, there exists a Rajchman self-conformal measure $\mu$ supported inside $[0,1]$ and generated by a $C^{\infty}$ IFS $\{S_{a} \colon [0,1]\to [0,1]\}_{a\in \A}$, such that for every $(a_{1},\dotsc,a_{k})\in \cup_{n=1}^{\infty}\A^{n}$ we have $(S_{a_{1}}\circ \cdots \circ S_{a_{k}})''(x)\neq 0$ apart from at most finitely many $x\in [0,1]$, and 
\[ \limsup_{\xi \to \infty} \frac{|\widehat{\mu}(\xi)|}{\phi(\xi)} > 0. \]
\end{thm}

Theorem~\ref{t:slowconformal} also shows that the result following from~\cite{ACWW25,AHW,BB25,BS} that self-conformal measures for analytic IFSs with a non-affine map have polynomial Fourier decay fails dramatically if `analytic' is replaced by `$C^{\infty}$.' 

The proof of Theorem~\ref{t:slowconformal} is one of the most challenging parts of the paper; the strategy can be summarised as follows. 
As we have discussed, under mild conditions pushforwards of self-similar measures are self-conformal measures. The measure $\mu$ in Theorem~\ref{t:slowconformal} will be constructed as the pushforward of a non-Rajchman self-similar measure on $[0,1]$ whose support contains $0$ by a map of the form $x + h(x)$, see Proposition~\ref{p:slowpushforward} for a technical statement. 
Here, $h$ will be $C^{\infty}$ and will vanish along with all its derivatives at $0$, but $h(x) > 0$ for all $x > 0$ (so $h$ is not analytic at $0$). Moreover, $h$ will decay to $0$ extremely rapidly in a way which depends on $\phi$, and in such a way that if $0 < x < y$ then even a relatively small separation between $x$ and $y$ will mean that $h''(x)$ is much smaller than $h''(y)$. 
We then consider the three regions $0 \leq x \leq x_1$, $x_1 < x < x_2$ and $x_2 \leq x$ separately, for some well-chosen $0 < x_1 < x_2$. Since $h$ is very small on $[0,x_1]$ and $\mu$ is a non-Rajchman missing digit measure, there will be some large frequency $\xi$ (depending on $x_1$ and tending rapidly to $\infty$ as $x_1 \to 0^+$ in a way which depends on $\phi$) such that $\widehat{\mu|_{[-x_1,x_1]}}(\xi) \approx \mu([-x_1,x_1])$. 
We ensure $x_1,x_2$ are close enough together that the contribution of the second term is small. We use quantitative Fourier decay results for nonlinear pushforwards of self-similar measures from our previous work~\cite{BB25}, together with the rapid decay of $h''$, to show that the contribution of the third term is small. 
Finally, we have to prove that we can choose $h$ in a way that the conformal maps generating $\mu$ satisfy the desired nonlinearity condition, which we do by randomising the construction of $h$ and showing that a typical choice works.

As in the self-similar case with Theorem~\ref{t:explicitselfsim}, we can give an explicit example of a self-conformal measure whose Fourier transform decays at least as slowly as $(\log \log \xi)^{-1}$.
\begin{thm}\label{t:particularfunctionconf}
    Let $f \colon [0,1]\to \mathbb{R}$ be given by $f(x) = x + \exp(-\exp(x^{-2}))$ for $x \neq 0$ (and $f(0) = 0$), and let $\mu$ be the Cantor--Lebesgue measure for the IFS $\{T_{1}(x)=x/3,T_{2}(x)=(x+2)/3\}$ and the $(1/2,1/2)$ probability vector. Then $f \mu$ is a Rajchman self-conformal measure for an IFS of $C^{\infty}$ contractions $\{S_a \colon [0,f(1)] \to [0,f(1)]\}_{a \in \{1,2\}}$ such that for all $(a_1,\dotsc,a_k) \in \cup_{n=1}^{\infty} \{1,2\}^n$ we have $(S_{a_{1}}\circ \dotsb \circ S_{a_{k}})''(y) \neq 0$ apart from at at most finitely many $y \in [0,f(1)]$. Moreover, 
    \[ \limsup_{\xi \to \infty} ((\log \log \xi) |\widehat{f\mu}(\xi)|) > 0. \]
\end{thm}

We also have the following result whose proof is somewhat similar to the proof of Theorem~\ref{t:slowconformal}, though not quite as technical. 
This theorem shows that there exist nonlinear pushforwards of self-similar measures which are Rajchman but have slow Fourier decay along \emph{every} subsequence. 
Crucially, though, in this theorem the pushforward measures can no longer be realised as self-conformal measures, and they are zero-dimensional. 
\begin{thm}\label{t:slowzerodim}
For every function $\phi \colon [0,\infty) \to (0,1]$ such that $\phi(\xi) \to 0$ as $|\xi| \to \infty$, there exists a strictly increasing $C^{\infty}$ function $f \colon \mathbb{R} \to \mathbb{R}$ such that the following holds: if $\mu$ is any self-similar measure on $\mathbb{R}$ whose support contains $0$, then $f \mu$ is Rajchman but 
\begin{equation}\label{e:zerodimliminf}
    \liminf_{|\xi| \to \infty} \frac{|\widehat{f \mu}(\xi)|}{\phi(|\xi|)} > 0.
\end{equation}
\end{thm}
Note that in Theorem~\ref{t:slowzerodim}, $f$ depends on $\phi$ but not on the measure $\mu$. In this context we can also construct an explicit example of a pushforward whose Fourier transform decays at least as slowly as $(\log \log \xi)^{-1}$, using a very similar function to the one from Theorem~\ref{t:particularfunctionconf}.

\begin{thm}\label{t:particularfunctionzero}
    Let $f \colon [0,1]\to \R$ be given by $f(x) = \exp(-\exp(x^{-2}))$ for $x \neq 0$ (and $f(0) = 0$), and let $\mu$ be a self-similar measure whose support contains $0$. Then $f\mu$ is Rajchman but 
    \[ \liminf_{|\xi| \to \infty} ((\log \log |\xi|) |\widehat{f\mu}(\xi)|) > 0. \] 
\end{thm}

\begin{rem}
    Since self-similar and self-conformal measures have positive Frostman exponent (see  \cite[Proposition~2.2]{FL} and note that the proof generalises to self-conformal measures), there are various ways to see that we cannot hope to replace the $\limsup$ by a $\liminf$ in Theorem~\ref{t:slowconformal}. 
    For instance, the Frostman exponent (or $L^\infty$ dimension) is a lower bound for lower correlation dimension (or $L^2$ dimension), which is in turn described by the growth rate of $\int_{-R}^R |\widehat{\mu}(\xi)|^2 d\xi$ as $R \to \infty$, see \cite[Lemma~2.5]{FNW}. 
    The conclusions of Theorems~\ref{t:slowzerodim} and~\ref{t:particularfunctionzero}, however, do hold with $\liminf$; this is possible because the pushforward measures appearing in these statements have vanishing correlation dimension. 
\end{rem}

\begin{rem}
    Theorems~\ref{t:particularfunctionconf} and~\ref{t:particularfunctionzero} just use a particular example of the pushforward function $f$; one could modify them while staying in the $C^{\infty}$ class, and achieve decay that is considerably slower than $(\log \log |\xi|)^{-1}$. 
\end{rem}

\subsection*{Higher dimensions}
It is natural to ask about Fourier decay properties of self-similar and self-conformal measures in higher dimensions. 
\begin{rem}
    \begin{itemize}
        \item Theorem~\ref{t:slowselfsim} holds in higher dimensions, see Theorem~\ref{t:selfsimhigher} below. 
        Moreover, by the proof of Theorem~\ref{t:selfsimhigher}, one can obtain explicit examples of Rajchman homogeneous self-similar measures on $\R^d$ with very slow Fourier decay by taking the $d$-fold product of the measure from Theorem~\ref{t:explicitselfsim} (or any one of the measures from Remark~\ref{r:moreexplicitexamples}).
        \item For $d \geq 2$, since conformal maps on $\R^d$ are analytic, it seems plausible that every self-conformal measure in $\R^d$ which is not self-similar and gives zero mass to every proper submanifold of $\R^d$ should have polynomial Fourier decay. 
        Evidence for this is provided in~\cite{BKS,AHWplane,BYqualitative}. 
        Therefore there is not much hope for an analogue of Theorem~\ref{t:slowconformal} for self-conformal measures in higher dimensions. 
    \end{itemize}
\end{rem}

\subsection*{Structure of paper}

Section~\ref{sec:proofselfsim} relates to self-similar measures. 
In Section~\ref{ss:self-sim-main} we prove Theorem~\ref{t:slowselfsim} as a consequence of the stronger result Theorem~\ref{e:higherdimft} which holds in higher dimensions and for a topologically generic set of $t$. 
Section~\ref{ss:self-sim-explicit} relates to explicit self-similar measures, and we prove Theorem~\ref{t:explicitselfsim}. 
In Section~\ref{ss:slowequi} we prove Theorem~\ref{thm:slowequi}. 

Section~\ref{sec:overallproofnonlinear} relates to nonlinear fractal measures. 
We begin with preliminaries in Section~\ref{sec:prelim} before moving on to the proof of Theorem~\ref{t:pushisrajchman} in Section~\ref{sec:proofpushisrajchman}. 
Section~\ref{ss:provezerodim} contains the proof of Theorem~\ref{t:slowzerodim}, while Section~\ref{ss:proof-slow-conformal} contains the more technical proof of Theorem~\ref{t:slowconformal} as well as an auxiliary result (Proposition~\ref{p:slowpushforward}) on pushforwards of missing digit measures. 
Finally, in Section~\ref{ss:Examples} we prove Theorem~\ref{t:particularfunctionconf} which gives the explicit example of a self-conformal measure with slow Fourier decay, and we justify the remaining claims from Example~\ref{Example:self-similar image}.

\subsection*{Notation} We write $A \lesssim B$ to mean $A \leq CB$ for some uniform constant $C$. We also use the standard $f = O(g)$ notation to mean $f \leq Cg$ for some uniform constant $C$, and $f = o(g)$ to mean $f/g \to 0$. 
The class of twice continuously differentiable functions on $\R$ (or some subinterval of $\R$ depending on context) will be denoted $C^2$, while the class of infinitely differentiable functions is $C^{\infty}$. 
We sometimes use Knuth's up-arrow notation $\uparrow \uparrow$, see~\eqref{e:knuth}.

\section{Proofs of self-similar theorems}\label{sec:proofselfsim}

Consider the parameterised family of IFSs 
\[ \left\{S_{0}(x)=\frac{x}{10},\, S_{1}(x)=\frac{x+1}{10},\, S_{2}(x)=\frac{x+t}{10}\right\}, \] 
where $t\in [0,1]$. 
Throughout this section we let $\mu_{t}$ be the self-similar measure associated to this IFS and the uniform probability vector $(1/3,1/3,1/3)$. 

\subsection{Main proofs}\label{ss:self-sim-main}
Recall that a $G_{\delta}$ set is a countable intersection of open sets. 
Most of the work in this section is dedicated to proving the following result. 

\begin{prop}
	\label{prop:Main}
	Let $\phi \colon [0,\infty)\to (0,1]$ satisfy $\lim_{\xi\to\infty}\phi(\xi)=0$. Then there exists a dense $G_{\delta}$ set $A\subset [0,1]$ such that for all $t\in A$ we have \[\limsup_{\xi\to\infty}\frac{|\widehat{\mu_{t}}(\xi)|}{\phi(\xi)}>0.\]
\end{prop}
Combining Proposition~\ref{prop:Main} with a result of Br\'emont~\cite{Bre} implies the following stronger statement, which in turn implies Theorem~\ref{t:slowselfsim}. 

\begin{prop}\label{p:fullselfsim}
	Let $\phi \colon [0,\infty)\to (0,1]$ satisfy $\lim_{\xi\to\infty}\phi(\xi) = 0$. Then there exists a dense $G_{\delta}$ set $A\subset [0,1]$ such that for all $t\in A$ the measure $\mu_{t}$ is Rajchman and satisfies 
    \begin{equation*}
    \limsup_{\xi\to\infty}\frac{|\widehat{\mu_{t}}(\xi)|}{\phi(\xi)}>0.
    \end{equation*}
\end{prop}
\begin{proof}[Proof of Proposition~\ref{p:fullselfsim} assuming Proposition~\ref{prop:Main}]
Let $\phi$ be given and let $A$ be the corresponding dense $G_{\delta}$ set given by Proposition~\ref{prop:Main}. Let $A' \coloneqq A\cap (\mathbb{R}\setminus \mathbb{Q})$. 
Since the set of irrational numbers is a dense $G_{\delta}$ set, and the intersection of two dense $G_{\delta}$ sets is a dense $G_{\delta}$ set, to show that the desired conclusion holds it suffices to show that if $t\in A'$ then $\mu_{t}$ is Rajchman. 

A result of Br\'emont \cite[Theorem~2.5]{Bre} tells us that if $\mu_{t}$ is not Rajchman then there exists an affine map $g(x)=ax+b$ (with $a \in \R \setminus \{0\}$, $b \in \R$) such that $\{g^{-1}\circ S_i\circ g\}_{i=0}^{2}$ is in Pisot form. For such an affine map $g$, a quick calculation yields that 
\begin{align*}
	(g^{-1}\circ S_0\circ g)(x) &= \frac{x}{10}+\frac{b}{10a}-\frac{b}{a},\\
	(g^{-1}\circ S_1\circ g)(x) &= \frac{x}{10}+\frac{b}{10a}-\frac{b}{a}+\frac{1}{10a},\\
	(g^{-1}\circ S_2\circ g)(x) &= \frac{x}{10}+\frac{b}{10a}-\frac{b}{a}+\frac{t}{10a}.
\end{align*} 
We omit the definition of a Pisot IFS (see \cite[Definition~2.3]{Bre} for the definition) but merely remark that since the common contraction ratio $1/10$ is the reciprocal of an integer, if $\{g^{-1}\circ S_i\circ g\}_{i=0}^{2}$ were a Pisot IFS, then the translation parameters for each of the maps listed above would belong to $\mathbb{Q}$. 
Examining the first two equations it is clear that if $\{g^{-1}\circ S_i\circ g\}_{i=0}^{2}$ were in Pisot form then $a,b\in \mathbb{Q}$. Examining the third equation when equipped with this knowledge, we see that this in turn implies that $t\in\mathbb{Q}$. But this contradicts our assumption that $t\in A'$ and is therefore irrational. Thus $\{g^{-1}\circ S_i\circ g\}_{i=0}^{2}$ is never in Pisot form. So by Br\'emont's result it must be the case that $\mu_{t}$ is indeed Rajchman.
\end{proof}	

We can also use Proposition~\ref{p:fullselfsim} to prove a stronger statement that also holds in higher dimensions. 
The Fourier transform of a Borel probability measure $\mu$ on $\RR^d$ is defined by 
\begin{equation}\label{e:higherdimft}
\widehat{\mu} \colon \RR^d \to \C, \qquad \widehat{\mu}(\bxi) = \int_{\RR^d} e^{2 \pi i \langle \bxi,\bx \rangle} d \mu(\bx), 
\end{equation} 
where $\langle \cdot, \cdot \rangle$ is the Euclidean inner product. We say that a Borel probability measure on $\R^{d}$ is Rajchman if $\lim_{\|\bxi\|\to\infty}\widehat{\mu}(\bxi)=0$.

\begin{thm}\label{t:selfsimhigher}
    Let $d \geq 1$ be an integer and let $\phi \colon [0,\infty)\to (0,1]$ satisfy $\lim_{\xi\to\infty}\phi(\xi) = 0$. Then there exists a dense $G_{\delta}$ set $A\subset [0,1]$ such that for all $t\in A$, the measure $\mu_{t}^d \coloneqq \overbrace{\mu_{t}\times \cdots \times \mu_{t}}^{\text{d terms}}$ (which is a homogeneous self-similar measure for the $d$-fold product IFS) is Rajchman and satisfies 
    \[ \limsup_{\xi \to \infty} \frac{|\widehat{\mu_t^d}((\xi,0,\dotsc,0))|}{\phi(\xi)} >0. \]
\end{thm}

\begin{proof}[Proof of Theorem~\ref{t:selfsimhigher} assuming Proposition~\ref{p:fullselfsim}]
Let $A$ be as in Proposition~\ref{p:fullselfsim} and fix $t \in A$. 
Then for all $\xi_1,\dotsc,\xi_d \in \RR$, \eqref{e:higherdimft} and the product structure of $\mu_t^d$ give that 
\begin{equation}\label{e:prodofft}
\widehat{\mu_t^d}((\xi_1,\dotsc,\xi_d)) = \prod_{i=1}^d \widehat{\mu_t}(\xi_i).
\end{equation}
Also $\widehat{\mu_t}(0) = 1$, so by Proposition~\ref{p:fullselfsim}, for all $\xi \in \RR$, 
\[ \limsup_{\xi \to \infty} \frac{|\widehat{\mu_t^d}((\xi,0,\dotsc,0))|}{\phi(\xi)} = \limsup_{\xi \to \infty} \frac{|\widehat{\mu_t}(\xi)|}{\phi(\xi)} >0. \]
Moreover,~\eqref{e:prodofft}, the fact that $|\widehat{\mu_t}(\xi)| \leq 1$ for all $\xi \in \RR$, and the Rajchman property of $\mu_t$ imply that 
\[ 
|\widehat{\mu_t^d}((\xi_1,\dotsc,\xi_d))| \leq \min_{1 \leq i \leq d} |\widehat{\mu_t}(\xi_i)| \leq \sup_{|\xi| \geq \max_{1 \leq i \leq d} |\xi_i|} |\widehat{\mu_t}(\xi)| \xrightarrow[|(\xi_1,\dotsc,\xi_d)| \to \infty]{} 0. 
\]
Therefore $\mu_t^d$ is Rajchman. 
\end{proof}

What remains is to prove Proposition~\ref{prop:Main}. 
We will regularly use the fact that $\mu_{t} = \pi_{t} \mathbb{P}$ for each $t\in [0,1]$, where $\mathbb{P}$ is the uniform Bernoulli measure on $\{0,1,2\}^{\N}$ and $\pi_{t} \colon \{0,1,2\}^{\N}\to \mathbb{R}$ is given by 
\begin{equation}\label{e:pointinselfsim}
    \pi_{t}((a_j)_{j=1}^{\infty}) = \sum_{j \geq 1 : a_{j}\in \{0,1\}}\frac{a_j}{10^{j}}+\sum_{j \geq 1 : a_{j}=2}\frac{t}{10^{j}}.
\end{equation}
Yet another equivalent way of defining the homogeneous self-similar measure $\mu_{t}$ is by the infinite convolution formula
\[ \mu_t = *_{j=1}^{\infty} \frac{1}{3} (\delta_0 + \delta_{1/10^j} + \delta_{t/10^j}). \] 
This yields the following formula, which we will also use regularly: 
\[ \widehat{\mu_{t}}(\xi)=\prod_{j=1}^{\infty}\frac{1}{3}\left(1+e^{2\pi i\xi/10^{j}}+e^{2\pi it/10^{j}}\right).\]

We now collect some useful lemmas.
\begin{lem}
	\label{lem:approximation lem}
	Let $N\in\mathbb{N}$ and $\kappa>0$. If $|t-t'|<\frac{\kappa}{2\pi N}$ then $|\widehat{\mu_{t}}(\xi)-\widehat{\mu_{t'}}(\xi)|<\kappa$ for all $\xi\in [-N,N]$.
\end{lem}
\begin{proof}
Let $N,\kappa,t$ and $t'$ be as in the statement of the lemma. For $\xi\in [-N,N]$ we observe
\begin{align*}
&\phantom{=} \ |\widehat{\mu_{t}}(\xi) - \widehat{\mu_{t'}}(\xi)|\\
&= \left|\int e^{2\pi i \xi x} d\mu_{t}(x)-\int e^{2\pi i \xi x} d\mu_{t'}(x)\right|\\
&= \left|\int e^{2\pi i \xi \left(\sum_{j:a_j\in \{0,1\}}\frac{a_{j}}{10^{j}}+\sum_{j:a_j=2}\frac{t}{10^{j}}\right)}-e^{2\pi i \xi \left(\sum_{j:a_j\in \{0,1\}}\frac{a_{j}}{10^{j}}+\sum_{j:a_j=2}\frac{t'}{10^{j}}\right)} d\mathbb{P}((a_j)_j)\right|\\
&\leq \int \left|e^{2\pi i \xi \left(\sum_{j:a_j\in \{0,1\}}\frac{a_{j}}{10^{j}}+\sum_{j:a_j=2}\frac{t}{10^{j}}\right)}-e^{2\pi i \xi \left(\sum_{j:a_j\in \{0,1\}}\frac{a_{j}}{10^{j}}+\sum_{j:a_j=2}\frac{t'}{10^{j}}\right)}\right| d\mathbb{P}((a_j)_j)\\
&\leq \int 2\pi  |\xi| \left|\left(\sum_{j:a_j\in \{0,1\}}\frac{a_{j}}{10^{j}}+\sum_{j:a_j=2}\frac{t}{10^{j}}\right)-\left(\sum_{j:a_j\in \{0,1\}}\frac{a_{j}}{10^{j}}+\sum_{j:a_j=2}\frac{t'}{10^{j}}\right)\right| d\mathbb{P}((a_j)_j)\\
&\leq 2\pi |\xi||t-t'|\\
&\leq \kappa.
\end{align*}
In the third from last line we used that $x\mapsto e^{2\pi i \xi x}$ is Lipschitz with Lipschitz constant $2\pi |\xi|$. In the final line we used that $\xi\in [-N,N]$ and that $|t-t'|\leq \frac{\kappa}{2\pi N}$. 
\end{proof}

\begin{lem}\label{lem:Decay lower bound}
There exists $c>0$ such that the following holds. Let $n\in \mathbb{N}$ and $t=p/10^{n}$ for $p\in \mathbb{N}$ satisfying $0\leq p\leq 10^{n}$. Then for all integers $L > n$ we have 
\[ |\widehat{\mu_{t}}(10^{L})|\geq c\left(\frac{1}{3}\right)^{n}.\]
\end{lem}
\begin{proof}
For $t=p/10^{n}$ we have
\begin{align*}
	\widehat{\mu_{t}}(10^{L}) &= \prod_{j=1}^{\infty}\frac{1}{3}\left(1+e^{2\pi i10^{L-j}}+e^{2\pi ip10^{L-j-n}}\right)\\
	&= \prod_{j=1}^{L-n}\frac{1}{3}\left(1+e^{2\pi i10^{L-j}}+e^{2\pi ip10^{L-j-n}}\right)\\*
	&\phantom{-} \times \prod_{j=L-n+1}^{L}\frac{1}{3}\left(1+e^{2\pi i10^{L-j}}+e^{2\pi ip10^{L-j-n}}\right)\\*
	&\phantom{-} \times \prod_{j=L+1}^{\infty}\frac{1}{3}\left(1+e^{2\pi i10^{L-j}}+e^{2\pi ip10^{L-j-n}}\right).
\end{align*}

We will treat each of the three terms above individually. For the first term we have
\begin{equation}
	\label{Eq:easy}
	\prod_{j=1}^{L-n}\frac{1}{3}\left(1+e^{2\pi i10^{L-j}}+e^{2\pi ip10^{L-j-n}}\right)=1
\end{equation}since $10^{L-j},10^{L-j-n}\in \mathbb{N}$ whenever $1\leq j\leq L-n$. For the second term we have 
\begin{equation}
	\label{Eq:easy2}
\left|	\prod_{j=L-n+1}^{L}\frac{1}{3}\left(1+e^{2\pi i10^{L-j}}+e^{2\pi ip10^{L-j-n}}\right)\right|\geq \left(\frac{1}{3}\right)^{n}
\end{equation} 
since 
\[ \left|1+e^{2\pi i10^{L-j}}+e^{2\pi ip10^{L-j-n}}\right|=|2+e^{2\pi ip10^{L-j-n}}|\geq 1 \] 
for $j\leq L$. For the third term we first remark that for $j\geq L+1$ we have
\[ |1-e^{2\pi i10^{L-j}}|\leq 2\pi 10^{L-j} \] 
and 
\[ |1-e^{2\pi ip10^{L-j-n}}|\leq 2\pi p10^{L-j-n}\leq 2\pi 10^{L-j}.\] 
Here we have used that $x \mapsto e^{2\pi i x}$ is Lipschitz with Lipschitz constant $2\pi$, and in the final inequality we used that $0\leq p\leq 10^{n}$. Consequently, for the third term we have 
\begin{align*}
	\left|\prod_{j=L+1}^{\infty}\frac{1}{3}\left(1+e^{2\pi i10^{L-j}}+e^{2\pi ip10^{L-j-n}}\right)\right|&= \prod_{j=L+1}^{\infty}\frac{1}{3}\left|1+e^{2\pi i10^{L-j}}+e^{2\pi ip10^{L-j-n}}\right|\\
	&\geq \prod_{j=1}^{\infty} \left(1-\frac{4\pi}{3\cdot 10^{j}}\right).
	\end{align*}
We emphasise that the final term stated above is strictly positive and independent of $L$ and $n$; we denote it by $c$. Summarising, we have shown that 
\begin{equation}
	\label{Eq:Uniform bound}
	\left|\prod_{j=L+1}^{\infty}\frac{1}{3}\left(1+e^{2\pi i10^{L-j}}+e^{2\pi ip10^{L-j-n}}\right)\right|\geq c.
\end{equation}

Combining~\eqref{Eq:easy}, \eqref{Eq:easy2} and~\eqref{Eq:Uniform bound} yields 
\[ |\widehat{\mu_{t}}(10^{L})|\geq c\left(\frac{1}{3}\right)^{n}.\] 
This completes our proof. 
\end{proof}

Equipped with the lemmas above we can now prove Proposition~\ref{prop:Main}.

\begin{proof}[Proof of Proposition~\ref{prop:Main}]
Let $\phi \colon [0,\infty)\to (0,1]$ satisfying $\lim_{\xi\to\infty}\phi(\xi)=0$ be given and let $c$ be as in Lemma~\ref{lem:Decay lower bound}. Since $\phi(\xi) \to 0$ as $\xi \to \infty$, for each $n\in\mathbb{N}$ there exists some $L_{n}$ sufficiently large such that 
\begin{equation}\label{e:lnlarge}
    c\left(\frac{1}{3}\right)^{n}\geq 2\phi(10^{L_{n}}).
\end{equation}
By Lemma~\ref{lem:Decay lower bound} this inequality implies that if $t=p/10^{n}$ then 
\begin{equation}
	\label{Eq:L_N}
	|\widehat{\mu_{p/10^{n}}}(10^{L_{n}})|\geq 2\phi(10^{L_{n}}).
\end{equation} 

To each $n\in \mathbb{N}$ we now associate the set 
\begin{equation}\label{e:definexn} 
X_{n} = \bigcup_{0\leq p\leq 10^{n}}\left(\frac{p}{10^{n}} - \frac{\phi(10^{L_{n}})}{2\pi 10^{L_{n}}},\frac{p}{10^{n}} + \frac{\phi(10^{L_{n}})}{2\pi 10^{L_{n}}}\right).
\end{equation}
Applying Lemma~\ref{lem:approximation lem} with $N = 10^{L_n}$ and $\kappa = \phi(10^{L_{n}})$, if $t\in X_{n}$ then 
\[ |\widehat{\mu_{t}}(10^{L_{n}})-\widehat{\mu_{p/10^{n}}}(10^{L_{n}})|\leq \phi(10^{L_{n}}) \] 
for some $0 \leq p\leq 10^{n}$. It follows from this inequality and~\eqref{Eq:L_N} that if $t\in X_{n}$ then 
\begin{equation}\label{e:lowerboundselfsim}
    |\widehat{\mu_{t}}(10^{L_{n}})|\geq \phi(10^{L_{n}}). 
\end{equation}

Now define \[ A \coloneqq \bigcap_{m=1}^{\infty}\bigcup_{n=m}^{\infty} X_{n}. \] 
The set $A$ is a dense $G_{\delta}$ set by the Baire category theorem. Moreover, every $t\in A$ satisfies 
\[ |\widehat{\mu_{t}}(10^{L_{n}})|\geq \phi(10^{L_{n}}) \] 
for infinitely many $n$. 
Our result now follows. 
\end{proof}

\subsection{Explicit self-similar measures with slow Fourier decay}\label{ss:self-sim-explicit}
    Given a particular function $\phi$, it is not difficult to find an example of $t$ for which the conclusion of Proposition~\ref{p:fullselfsim} holds. Indeed, one can define $L_1,L_2,\dotsc$ by~\eqref{e:lnlarge} from the proof of Proposition~\ref{prop:Main}. Also, let $k_1 < k_2 < \dotsb$ be a (rapidly growing) sequence of positive integers defined inductively so that for all $n \geq 1$ we have $k_{n+1} > 2k_n$, and 
    \begin{equation}\label{e:knlarge}
    \frac{1}{10^{k_{n+1}}} < \frac{\phi(10^{L_{{k_n}}})}{4 \pi 10^{L_{k_n}}}.
    \end{equation} 
    Let 
    \[ t = \sum_{n=1}^{\infty} \frac{1}{10^{k_n}}. \]
    Then for all $n \geq 1$ one can write $\sum_{m=1}^{n} \frac{1}{10^{k_m}}$ in the form $p/10^{k_n}$ for some integer $0 \leq p \leq 10^{k_n}$. Also,  
    \[ t - \sum_{m=1}^{n} \frac{1}{10^{k_m}} = \sum_{m=n+1}^{\infty} \frac{1}{10^{k_m}} < \sum_{m=k_{n+1}}^{\infty} \frac{1}{10^{m}} < \frac{2}{10^{k_{n+1}}} < \frac{\phi(10^{L_{{k_{n}}}})}{2 \pi 10^{L_{k_{n}}}}, \]
    where we used~\eqref{e:knlarge} for the final inequality. 
    We have shown that for all integers $n \geq 1$ we have $t \in X_{k_{n}}$, and thus $|\widehat{\mu_{t}}(10^{L_{k_n}})|\geq \phi(10^{L_{k_n}})$ as in~\eqref{e:lowerboundselfsim}. 
    Moreover, since $k_{n+1} > 2k_n$ for all $n$, the decimal expansion of $t$ contains infinitely many $1$s separated by strings of zeros whose lengths increase without bound. The expansion is therefore not eventually periodic, which implies that $t \notin \mathbb{Q}$. The fact $\mu_{t}$ is Rajchman now follows from \cite[Theorem~2.5]{Bre} and an analogous argument to that given in the proof of Proposition~\ref{p:fullselfsim}.

Next, we implement this strategy in the case $\phi(\xi) = (\log \log \xi)^{-1}$ to prove Theorem~\ref{t:explicitselfsim}. Here $t$ has the additional benefit of a simple closed form.
\begin{proof}[Proof of Theorem~\ref{t:explicitselfsim}]
    First note that for the reason described above, the $t$ defined by~\eqref{e:explicitt} is irrational and so $\mu_{t}$ is Rajchman. 

    Now let $\phi(\xi) = (\log \log \xi)^{-1}$ for $\xi \geq 10$. 
    Recalling Knuth's notation from~\eqref{e:knuth}, for all $n \in \N$ let $L_{10 \uparrow \uparrow n} = 10 \uparrow \uparrow (n+2)$. 
    For this choice of $\phi$ and $L_{10 \uparrow \uparrow n}$, define the sets $X_{10 \uparrow \uparrow 1}, X_{10 \uparrow \uparrow 2},\dotsc$ as in~\eqref{e:definexn}. 
    Note that we are only interested in the integers $L_m$ and sets $X_m$ when $m$ is of the form $m = 10 \uparrow \uparrow n$ for some positive integer $n$. 
    For all $n$ sufficiently large, with $c$ as in Lemma~\ref{lem:Decay lower bound}, 
    \[ \phi(10^{L_{10 \uparrow \uparrow n}}) = \phi(10 \uparrow \uparrow (n+3)) \leq \frac{1}{10 \uparrow \uparrow (n+1)} < \frac{c}{2} \left(\frac{1}{3}\right)^{10 \uparrow \uparrow n}. \] 
    Thus~\eqref{e:lnlarge} is satisfied for all sufficiently large integers of the form $10 \uparrow \uparrow n$. 
    
    We also have the following lower bound for all integers $n$ sufficiently large: 
    \[  \phi(10^{L_{10 \uparrow \uparrow n}}) = \frac{1}{(\log \log 10) + (\log 10)\cdot (10 \uparrow \uparrow (n+1))} \geq \frac{1}{2 \log 10} \cdot \frac{1}{10 \uparrow \uparrow (n+1)}. \]
    Using this and fact that 
    \begin{equation}\label{e:particularknuth}
    10^{L_{10 \uparrow \uparrow (3n - 1)}} = 10 \uparrow \uparrow (3n+2)
    \end{equation} 
    gives that for all sufficiently large $n$, 
    \begin{align*}
    \frac{\phi(10^{L_{10 \uparrow \uparrow (3n - 1)}})}{2\pi 10^{L_{10 \uparrow \uparrow (3n - 1)}}} &\geq \frac{1}{2 \pi} \cdot \frac{1}{10 \uparrow \uparrow (3n + 2)} \cdot \frac{1}{2\log 10} \cdot \frac{1}{10 \uparrow \uparrow (3n)}  \\ 
    &> \frac{1}{(10 \uparrow \uparrow (3n+2))^2} \\
    &= \frac{1}{10^{2 \cdot (10 \uparrow \uparrow (3n+1))}} \\
    &> \frac{2}{10 \uparrow \uparrow (3n+3)} \\
    &> \sum_{m=n+1}^{\infty} \frac{1}{10 \uparrow \uparrow (3m)} \\
    &= t - \sum_{m=1}^{n} \frac{1}{10 \uparrow \uparrow (3m)}.
    \end{align*}
    This means that $t \in X_{10 \uparrow \uparrow (3n-1)}$ for all sufficiently large $n$. 
    The proof of Proposition~\ref{prop:Main} (using Lemmas~\ref{lem:approximation lem} and~\ref{lem:Decay lower bound}) gives $|\widehat{\mu_t}(10^{L_{10 \uparrow \uparrow (3n - 1)}})| \geq \phi(10^{L_{10 \uparrow \uparrow (3n - 1)}})$ for large enough $n$ as in~\eqref{e:lowerboundselfsim}. 
    Substituting~\eqref{e:particularknuth} gives 
    \[ |\widehat{\mu_t}(10 \uparrow \uparrow (3n+2))| \geq \phi(10 \uparrow \uparrow (3n+2)) \] 
    for all sufficiently large $n$, as required. 
\end{proof}

Regarding Remark~\ref{r:moreexplicitexamples}, note that if $t = \sum_{n=1}^{\infty} \frac{1}{10 \uparrow \uparrow (kn)}$, then almost the same proof works with $\phi(\xi) = (\underbrace{\log \circ \dotsb \circ \log}_{k-1 \mbox{ times}} (\xi))^{-1}$ and $L_n = 10 \uparrow \uparrow (n + k - 1)$.

\subsection{Proof of Theorem~\ref{thm:slowequi}}\label{ss:slowequi}

\begin{proof}[Proof of Theorem~\ref{thm:slowequi}]
    The definition of $\psi(1), \psi(2),\dotsc$ is as follows: 
    \begin{equation*}
    \psi(n) = 
    \begin{cases}
    \frac{1}{10 \uparrow \uparrow 3} & \mbox{ for } 1 \leq n \leq (10 \uparrow \uparrow 5) - (10 \uparrow \uparrow 2), \\*
    \frac{1}{10 \uparrow \uparrow (3j)} & \parbox[t]{.6\textwidth}{ for $j \geq 2, (10 \uparrow \uparrow (3j-1)) - (10 \uparrow \uparrow (3j-4)) < n \leq (10 \uparrow \uparrow (3j+2)) - (10 \uparrow \uparrow (3j-1))$.}
    \end{cases}
    \end{equation*}
    Clearly for all integers $N > 1$ we have $\Psi(N) < N$, hence $\limsup_{N\to\infty}\frac{\log \Psi(N)}{\log N} \leq 1$. 
    For all $j \geq 2$, 
    \begin{align*}
    \Psi(10 \uparrow \uparrow (3j-1)) &> \frac{1}{10 \uparrow \uparrow (3j-3)} \cdot ((10 \uparrow \uparrow (3j-1)) - (10 \uparrow \uparrow (3j-4))) \\
    &> \frac{10 \uparrow \uparrow (3j-1)}{10 \uparrow \uparrow (3j-2)}.
    \end{align*}
    It follows that 
    \[ \limsup_{j \to \infty} \frac{\log\Psi(10 \uparrow \uparrow (3j-1))}{\log(10 \uparrow \uparrow (3j-1))} \geq 1 - \liminf_{j \to \infty} \frac{\log (10 \uparrow \uparrow (3j-2))}{\log(10 \uparrow \uparrow (3j-1))} = 1. \]
    Therefore $\limsup_{N\to\infty}\frac{\log \Psi(N)}{\log N}=1$ and~\eqref{i:psigrowsquickly} is satisfied. 

    To complete our proof it suffices to construct a $\gamma \in (0,1)$ for which 
    \[ \dist(10^{n-1}x-\gamma,\mathbb{Z})>\psi(n) \] 
    for all $n\in \N$ and $x \in X_{t}$. We will define $(\gamma_n)_{n=1}^{\infty} \in \{0,\dotsc,9\}^{\N}$ inductively and the decimal expansion of $\gamma$ will be $\gamma = 0.\gamma_1 \gamma_2\dots$. 
    There are many different choices we can make, but for concreteness let $\gamma_1 = \dotsb = \gamma_{10 \uparrow \uparrow 2} = 5$. 
    Let $x \in X_t$ be arbitrary and write its decimal expansion as $0.x_1 x_2 \dots$. From~\eqref{e:explicitt} and~\eqref{e:pointinselfsim}, we see that there exists $(a_{k})\in \{0,1,2\}^{\N}$ such that
    \begin{align*}
        x &= \sum_{k:a_{k}=1}\frac{1}{10^{k}}+\sum_{k:a_{k}=2}\frac{t}{10^{k}}\\
        &= \sum_{k:a_{k}=1}\frac{1}{10^{k}}+\frac{1}{10 \uparrow \uparrow 3}\sum_{k:a_{k}=2}\frac{1}{10^{k}}+\sum_{l=2}^{\infty}\frac{1}{10\uparrow\uparrow (3l)}\sum_{k:a_{k}=2}\frac{1}{10^{k}}.
    \end{align*} 
    Therefore \[ 0 < x-\sum_{1\leq k< 10\uparrow\uparrow 5:a_{k}=1}\frac{1}{10^{k}} - \frac{1}{10 \uparrow \uparrow 3}\sum_{1\leq k< 10\uparrow\uparrow 5 - 10\uparrow\uparrow 2:a_{k}=2}\frac{1}{10^{k}}<\frac{1}{10^{(10\uparrow\uparrow 5) - 1}}.\] 
    Hence $x_{n}\in \{0,1,2\}$ for $1 \leq n < 10 \uparrow \uparrow 5$. 

   Using the fact that $x_{n}\in \{0,1,2\}$ for all $1 \leq n < 10 \uparrow \uparrow 5$ and $x\in X_{t}$, we see that 
    \[ \dist(10^{n-1}x-\gamma,\mathbb{Z})>\psi(n)  \]
    for all $1 \leq n \leq (10 \uparrow \uparrow 5) - (10 \uparrow \uparrow 2)$ and $x\in X_{t}$ whenever $\gamma$ satisfies $\gamma_1 = \dotsb = \gamma_{10 \uparrow \uparrow 2} = 5$. 
    Assume now that for some positive integer $j$ we have defined $\gamma_1,\dotsc,\gamma_{10 \uparrow \uparrow (3j-1)}$ and shown that for every $\gamma$ whose decimal expansion begins with $\gamma_1,\dotsc,\gamma_{10 \uparrow \uparrow (3j-1)}$ we have 
    \[ \dist(10^{n-1}x-\gamma,\mathbb{Z})>\psi(n) \]
    for all $1 \leq n \leq (10 \uparrow \uparrow (3j+2)) - (10 \uparrow \uparrow(3j-1))$ and $x \in X_t$. Let $n$ be such that 
    \begin{equation}\label{e:inductiverangeofn}
    (10 \uparrow \uparrow (3j+2)) - (10 \uparrow \uparrow (3j-1)) < n \leq (10 \uparrow \uparrow (3j+5)) - (10 \uparrow \uparrow (3j+2)). 
    \end{equation}
    Write $t = t^{(1)} + t^{(2)}$, where 
    \[ t^{(1)} = \sum_{k=1}^{j+1} \frac{1}{10 \uparrow \uparrow (3k)}; \qquad t^{(2)} = \sum_{k=j+2}^{\infty} \frac{1}{10 \uparrow \uparrow (3k)}. \]
    Let $x\in X_{t}$, and let $(a_{k})\in \{0,1,2\}^{\N}$ be such that $\pi_{t}((a_k))=x$. By~\eqref{e:explicitt} and \eqref{e:pointinselfsim}, we can write $x = x^{(1)} + x^{(2)} + x^{(3)}$, where 
    \begin{align*}
    x^{(1)} &\coloneqq \sum_{1 \leq k < n : a_k = 1}\frac{1}{10^{k}} + \sum_{1 \leq k < n - (10 \uparrow \uparrow (3j+2)) : a_k = 2}\frac{t^{(1)}}{10^{k}}, \\
    x^{(2)} &\coloneqq \sum_{n \leq k < n + (10 \uparrow \uparrow (3j+2)) : a_k = 1}\frac{1}{10^{k}} + \sum_{n - (10 \uparrow \uparrow (3j+2)) \leq k < n + (10 \uparrow \uparrow (3j+2)) : a_k = 2}\frac{t^{(1)}}{10^{k}}, \\
    x^{(3)} &\coloneqq \sum_{k \geq n + (10 \uparrow \uparrow (3j+2)) : a_k = 1}\frac{1}{10^{k}} + \sum_{k \geq n + (10 \uparrow \uparrow (3j+2)) : a_k = 2}\frac{t^{(1)}}{10^{k}} + \sum_{k \geq 1 : a_k = 2} \frac{t^{(2)}}{10^{k}}.
    \end{align*}

   Now, for $i \in \{1,2,3\}$, let $x^{(i)}$ have decimal expansion $0.x^{(i)}_1 x^{(i)}_2 \dots$. 
   We have defined $x^{(1)}$ so that $x^{(1)}=\frac{p}{10^{n-1}}$ for some $p\in \mathbb{N}$. This implies that $10^{n-1}x \mod 1$ does not depend upon $x^{(1)}$. 
   The following counting argument for the possible decimal expansions of $x^{(2)}$ is the crux of the proof. 
   Let 
   \begin{align*} 
   B_{j}\coloneqq \Bigg\{&\sum_{0 \leq k < (10 \uparrow \uparrow (3j+2)) : b_k = 1}\frac{1}{10^{k+1}} + \sum_{ - (10 \uparrow \uparrow (3j+2)) \leq k <  (10 \uparrow \uparrow (3j+2)): b_k = 2}\frac{t^{(1)}}{10^{k+1}} \\*
   &:(b_{k})_{k=- (10 \uparrow \uparrow (3j+2))}^{(10 \uparrow \uparrow (3j+2))-1}\in \{0,1,2\}^{2(10 \uparrow \uparrow (3j+2))}\Bigg\}.
   \end{align*}
   The key observation is that for any $n$ satisfying~\eqref{e:inductiverangeofn} and $x\in X_{t}$ we have $10^{n-1}x^{(2)}\in B_{j}$. 
   Crucially though, the set $B_{j}$ does not depend upon $x$ or $n$ and contains at most $9^{10 \uparrow \uparrow (3j+2)}$ many elements. 
   This means that there are at most $9^{10 \uparrow \uparrow (3j+2)}$ possible values the string $(x^{(2)}_n, x^{(2)}_{n+1}, \dotsc, x^{(2)}_{n + (10 \uparrow \uparrow (3j+2)) - 1})$ can take. 
    But 
    \[ \# \{0,\dotsc,9\}^{(10 \uparrow \uparrow (3j+2)) - (10 \uparrow \uparrow (3j-1))} =  \frac{10 \uparrow \uparrow (3j+3)}{10 \uparrow \uparrow (3j)} > 1000\cdot  9^{10 \uparrow \uparrow (3j+2)}. \] Therefore there exists \[ (\gamma_1',\dotsc,\gamma_{(10 \uparrow \uparrow (3j+2)) - (10 \uparrow \uparrow (3j-1))}') \in \{0,\dotsc,9\}^{(10 \uparrow \uparrow (3j+2)) - (10 \uparrow \uparrow (3j-1))} \] such that for all $x \in X_t$ and $n$ satisfying~\eqref{e:inductiverangeofn} we have
    \[\left|\sum_{k=1}^{10\uparrow\uparrow (3j+2)}\frac{x_{n+k-1}^{(2)}}{10^{k}} - \sum_{k=1}^{10\uparrow\uparrow (3j-1)}\frac{\gamma_{k}}{10^{k}}-\sum_{k=1}^{10\uparrow\uparrow (3j+2)-10\uparrow\uparrow (3j-1)}\frac{\gamma_{k}'}{10^{k+10\uparrow\uparrow (3j-1)}}\right|\geq \frac{100}{10\uparrow\uparrow (3j+3)}.\] 
    This in turn implies that if $\gamma$ is any number whose decimal expansion begins with $\gamma_1,\dotsc,\gamma_{10 \uparrow \uparrow (3j-1)}, \gamma_1',\dotsc,\gamma_{(10 \uparrow \uparrow (3j+2)) - (10 \uparrow \uparrow (3j-1)) }'$ then 
    \begin{equation}
    \label{eq:gammabound}
    \left|\sum_{k=1}^{\infty}\frac{x_{n+k-1}^{(2)}}{10^{k}}-\gamma\right|\geq \frac{50}{10\uparrow\uparrow (3j+3)}
    \end{equation}
    for all $x\in X_{t}$ and $n$ satisfying~\eqref{e:inductiverangeofn}. 
    
    Turning now to $x^{(3)}$, if we inspect the definitions of $t^{(1)}$, $t^{(2)}$ and use the fact that $n$ satisfies~\eqref{e:inductiverangeofn}, we see that each of the three terms in the definition of $x^{(3)}$ is bounded above by $\frac{1}{10^{n-1} \cdot (10\uparrow\uparrow(3j+3))}$. 
    This implies that 
    \[ x^{(3)}\leq \frac{3}{10^{n-1} \cdot (10\uparrow\uparrow(3j+3))}, \]
    and hence 
    \begin{equation}
    \label{eq:x3bound}
    \sum_{k=1}^{\infty}\frac{x^{(3)}_{n+k-1}}{10^{k}}\leq \frac{3}{10\uparrow\uparrow (3j+3)}.
    \end{equation}

    It now follows from \eqref{eq:gammabound}, \eqref{eq:x3bound}, and the fact $x^{(1)} = \frac{p}{10^{n-1}}$ for some $p\in \mathbb{N}$, that for any $x\in X_{t}$ and $n$ satisfying~\eqref{e:inductiverangeofn}, if $\gamma$ is any number whose decimal expansion begins with $\gamma_1,\dotsc,\gamma_{10 \uparrow \uparrow (3j-1)}, \gamma_1',\dotsc,\gamma_{(10 \uparrow \uparrow (3j+2)) - (10 \uparrow \uparrow (3j-1)) }'$ then 
    \begin{align}\label{e:bigslowequicalc}
    \begin{split}
        \dist(10^{n-1}x-\gamma,\Z) &= \dist(10^{n-1}(x^{(1)}+x^{(2)}+x^{(3)})-\gamma,\Z)\\
        &= \dist(10^{n-1}(x^{(2)}+x^{(3)})-\gamma,\Z)\\
        &= \dist\left(\sum_{k=1}^{\infty}\frac{x^{(2)}_{n+k-1}}{10^{k}}+\sum_{k=1}^{\infty}\frac{x^{(3)}_{n+k-1}}{10^{k}}-\gamma,\Z\right )\\
        &\geq \min\left\{\left|\sum_{k=1}^{\infty}\frac{x^{(2)}_{n+k-1}}{10^{k}}+\sum_{k=1}^{\infty}\frac{x^{(3)}_{n+k-1}}{10^{k}}-\gamma\right|,0.2\right\}\\
        &\geq \min\left\{\left|\sum_{k=1}^{\infty}\frac{x^{(2)}_{n+k-1}}{10^{k}}-\gamma\right |-\frac{3}{10\uparrow\uparrow (3j+3)},0.2\right\}\\
        &> \frac{1}{10\uparrow\uparrow (3j+3)}\\
        &= \psi(n).
        \end{split}
    \end{align}
    In the first inequality above we used that $\gamma_{1}=5$ and \eqref{eq:x3bound} which together imply that 
    \[ \left|\sum_{k=1}^{\infty}\frac{x^{(2)}_{n+k-1}}{10^{k}}+\sum_{k=1}^{\infty}\frac{x^{(3)}_{n+k-1}}{10^{k}}-\gamma-p\right|\geq 0.2 \qquad \mbox{ for all } p\in \Z\setminus \{0\}. \]
    In the final equality we used the definition of $\psi$ and the fact that $n$ satisfies~\eqref{e:inductiverangeofn}. 
    We now define 
    \[ \gamma_{1},\dotsc, \gamma_{10\uparrow\uparrow (3j+2)} = \gamma_1,\dotsc,\gamma_{10 \uparrow \uparrow (3j-1)}, \gamma_1',\dotsc,\gamma_{(10 \uparrow \uparrow (3j+2)) - (10 \uparrow \uparrow (3j-1)) }'.\] By~\eqref{e:bigslowequicalc} and our inductive assumption for $\gamma_1,\dotsc,\gamma_{10 \uparrow \uparrow (3j-1)}$, if $\gamma$ has a decimal expansion that begins with $\gamma_{1},\dotsc, \gamma_{10\uparrow\uparrow (3j+2)}$, then for all $1\leq n  \leq (10 \uparrow \uparrow (3j+5)) - (10 \uparrow \uparrow (3j+2))$ and $x\in X_{t}$ we have 
    \[\dist(10^{n-1}x-\gamma,\Z)>\psi(n).\]
    
    It is clear that the above argument can be repeated indefinitely, yielding a sequence $(\gamma_{n})_{n=1}^{\infty}$. Taking $\gamma$ to be the point whose decimal expansion equals $(\gamma_{n})_{n=1}^{\infty}$, it follows from the above that 
    \[\dist(10^{n-1}x-\gamma,\Z)>\psi(n)\] 
    for all $n\in \N$. This completes our proof. 
    \end{proof}

\begin{rem}
    By replacing $10 \uparrow \uparrow (3n)$ in \eqref{e:explicitt} with an even more rapidly growing sequence of powers of $10$, one could of course make $(\psi(n))_{n=1}^{\infty}$ decay to $0$ even more slowly than in the proof of Theorem~\ref{thm:slowequi}, while still ensuring that $R(x,N;\gamma,\psi,10) = 0$ for all $N \in \N$ and all $x$ in the support of the self-similar measure.
\end{rem}

\section{Proof of results for nonlinear measures}\label{sec:overallproofnonlinear}

\subsection{Preliminaries}\label{sec:prelim}
In the proofs of Proposition~\ref{p:slowpushforward} and Theorem~\ref{t:slowzerodim} we will make use of results establishing polynomial Fourier decay for nonlinear pushforwards of self-similar measures. In a previous paper we proved the following (see \cite[Theorem~1.1]{ACWW25} for a different but related statement). 
\begin{thm}[\cite{BB25}, Corollary~1.5]
\label{t:originalbb}
Let $\mu$ be a self-similar measure on $[0,1]$. Then there exist $\eta, \kappa, C > 0$ such that the following holds. For all $C^2$ functions $f \colon [0,1] \to \RR$ which satisfy $f''(x) \neq 0$ for all $x \in [0,1]$, we have 
\begin{align*}
|\widehat{f \mu}(\xi)| \leq C &\left(1 + \max_{x \in [0,1]} |f'(x)| + (\max_{x \in [0,1]} |f'(x)|)^{-\kappa} + \max_{x \in [0,1]} |f''(x)| \right) \\*
&\times \left( 1 + (\min_{x \in [0,1]} |f''(x)|)^{-\kappa} \right) |\xi|^{-\eta} 
\end{align*} 
for all $\xi\neq 0$.
\end{thm}
From the proof of Theorem~\ref{t:originalbb} one in fact gets the following statement. 
\begin{thm}\label{t:bakerbanaji}
Let $\mu$ be a self-similar measure on $[0,1]$. 
Let $\mathcal{F}$ be a family of $C^2$ functions $[0,1] \to \RR$ and assume that there exist $C_1,C_2 > 0$ such that 
\begin{equation}
\label{e:C1C2 inequalities}
    \sup_{f \in \mathcal{F}} \max_{x \in [0,1]} |f'(x)| \leq C_{1} \qquad \mbox{ and } \qquad \sup_{f \in \mathcal{F}} \max_{x \in [0,1]} |f''(x)| \leq C_{2}.
\end{equation}
Assume moreover that $f''(x) \neq 0$ for all $f \in \mathcal{F}$ and $x \in [0,1]$. 
Then there exist $\eta,C>0$ such that for all $f \in \mathcal{F}$ and $\xi\neq 0$ we have 
\begin{equation}\label{e:firstsecondderivbounds} 
|\widehat{f\mu}(\xi)| \leq C (\min_{x\in [0,1]}|f''(x)|)^{-1}|\xi|^{-\eta}.
\end{equation}
We emphasise that $\eta$ and $C$ can depend upon $\mu$ and $\mathcal{F}$ but are independent of $f$.
\end{thm} 

\begin{proof}
Inspecting the proof of Theorem~\ref{t:originalbb} given in \cite{BB25}, we see that $\kappa$ can be taken to be any sufficiently small positive number (though we may have to reduce $\eta$ in order to do this). Crucially, if the stated inequality in Theorem~\ref{t:originalbb} holds for some $\kappa<1/2$, then at the cost of increasing the value of the constant $C$ we can in fact take $\kappa=1/2$. Combining the inequalities from~\eqref{e:C1C2 inequalities} with the choice of $\kappa=1/2$, we see that Theorem~\ref{t:originalbb} implies that for all $f\in\mathcal{F}$ there exists $C_{\mathcal{F},\mu}>0$ depending upon $\mathcal{F}$ and $\mu$ such that 
\begin{equation}\label{e:nonlinpushintermediatestep}
|\widehat{f\mu}(\xi)|\leq C_{\mathcal{F},\mu} \left(1+(\max_{x \in [0,1]} |f'(x)|)^{-1/2}  \right) \left( 1 + (\min_{x \in [0,1]} |f''(x)|)^{-1/2} \right) |\xi|^{-\eta} 
\end{equation}
for all $\xi\neq 0$. 
By a mean value theorem argument, we know that for all $f$, 
\[ \max_{x \in [0,1]} |f'(x)| \geq \frac{1}{2} \min_{x \in [0,1]} |f''(x)|. \] 
Substituting this into~\eqref{e:nonlinpushintermediatestep} gives that for all $f \in \mathcal{F}$ and $\xi\neq 0$ we have 
\[ |\widehat{f\mu}(\xi)|\leq 2C_{\mathcal{F},\mu} \left( 1 + (\min_{x \in [0,1]} |f''(x)|)^{-1/2} \right)^{2} |\xi|^{-\eta}. \] 
By our uniform upper bound on the second derivatives, we know that there exists $C_{3}>0$ such that for all $f\in\mathcal{F}$ we have $(\min_{x\in [0,1]}|f''(x)|)^{-1/2}\geq C_{3}$. Thus, the constant $1$ term appearing within the bracket can be absorbed into our second derivative term at the cost of potentially increasing our multiplicative constant $C_{\mathcal{F},\mu}$. Once this constant term has been absorbed, our result immediately follows after squaring what remains in the bracket. 
\end{proof}

For the deduction of Theorem~\ref{t:slowconformal} from Proposition~\ref{p:slowpushforward}, we will use the following standard fact which says that smooth diffeomorphic images of self-similar measures are self-conformal measures. 
\begin{lem}\label{lem:pushforwardconformal}
    Let $\mu$ be a self-similar measure on $\R$. Let $f\colon \R \to \R$ be $C^{\rho}$ (here $\rho \in \{1,2,\dotsc \}$ or $\rho=\infty$ or $\rho=\omega$) with $f'(x) \neq 0$ for all $x \in [\min \supp(\mu),\max \supp(\mu)]$. 
    Then $f \mu$ is a self-conformal measure for some $C^{\rho}$ conformal IFS of the form \[ \left\{f\circ R_{a_{1}}\circ \cdots \circ R_{a_{k}}\circ f^{-1}\right\}_{(a_{1},\dotsc,a_{k})\in \A^{k}}\] for some $k\in \N$.  
\end{lem}

\begin{proof}
    Let $\{R_{a}\}_{a\in \mathcal{A}}$ and $(p_{a})_{a\in\A}$ be the IFS and probability vector corresponding to $\mu$.  By the inverse function theorem, $f$ has a $C^{\rho}$ inverse on some open neighbourhood of $[\min \supp(\mu),\max \supp(\mu)]$. 
    By iterating the IFS if necessary, we may assume without loss of generality that for each $a\in \A$ the contraction ratio of $R_{a}$ is at most 
    \begin{equation}\label{e:ratio0.9}
        0.9 \times \frac{\min_{x \in [\min \supp(\mu),\max \supp(\mu)]} |f'(x)|}{\max_{x \in [\min \supp(\mu),\max \supp(\mu)]} |f'(x)|}. 
    \end{equation}
    For each $a \in \A$ let $S_a \coloneqq f\circ R_{a} \circ f^{-1}$. 
    Then by the chain rule and~\eqref{e:ratio0.9} we have 
    \[ \max_{a \in \A} \max_{x \in f([\min \supp(\mu),\max \supp(\mu)])} |S_a'(x)| \leq 0.9. \] It also follows from the chain rule that each $S_{a}$ has non-vanishing derivative. 
    Thus $\{S_{a}\}_{a\in \A}$ is an IFS of $C^{\rho}$ conformal contractions on $f([\min \supp(\mu),\max \supp(\mu)])$. 
    For all Borel $A \subset \R$ we have 
	    \begin{align*} 
		f\mu(A) = \mu(f^{-1} (A)) = \sum_{a \in \A} p_{a} \mu(R_{a}^{-1}\circ f^{-1}(A)) &= \sum_{a \in \A} p_{a} \mu (f^{-1}\circ S_{a}^{-1}(A))\\
		&= \sum_{a \in \A} p_{a} S_{a}(f\mu)(A), 
		\end{align*}
    where we used the self-similarity of $\mu$ in the second equality and the definition of $S_a$ in the third.
    Since the self-conformal measure for a given IFS and probability vector is unique, $f \mu$ is the self-conformal measure for the IFS $\{S_a\}_{a\in \A}$ and probability vector $(p_{a})_{a\in \A}$. 
\end{proof}

We will need the following facts about self-similar measures.
\begin{prop}\label{prop:frostman}
Let $\mu$ be a self-similar measure on $\R$. Then there exist $0 < s_F \leq s_M < \infty$, such that for all $x \in \mbox{supp}(\mu)$, $y \in \R$ and $0 < r \leq 1$, we have 
\[  r^{s_M} \lesssim \mu(x-r,x+r) \qquad \mbox{ and } \qquad \mu(y-r,y+r) \lesssim r^{s_F}, \]
with implicit constants depending only on $\mu$.\footnote{$F$ stands for Frostman in light of Frostman's lemma; $M$ stands for Minkowski in light of~\cite{FFK}.}
\end{prop}
\begin{proof}
Let $r_i$ be the $i$th contraction ratio and $p_i$ the corresponding probability weight. It is not difficult to see that we can take $s_M = \max_i \frac{\log p_i}{\log r_i}$. 
The existence of $s_F$ follows from \cite[Proposition~2.2]{FL}. 
\end{proof}
In the special case of missing digit measures, it can be shown that the $s_{F}$ and $s_{M}$ appearing in Proposition~\ref{prop:frostman} coincide. 

We will also use the notion of an $r$-stopping of a self-similar measure. Given a self-similar IFS $\{S_{a} \colon [0,1] \to [0,1]\}_{a \in \mathcal{A}}$ and $r>0$, an $r$-stopping is a finite set $\mathcal{I}$ of finite words over the alphabet $\A$ with $\mu = \sum_{\a\in \mathcal{I}} p_{\a} S_{\a} \mu$ and such that each $\a$ satisfies 
\[ r\cdot \min_{a\in \A}\|S'_{a}\|_{\infty} \leq \mbox{diam}(S_{\a}([0,1])) \leq r. \]
Here we use the notation $S_{\a} \coloneqq S_{a_1} \circ S_{a_2} \circ\cdots \circ S_{a_{n}}$ and $p_{\a} \coloneqq p_{a_{1}} p_{a_{2}} \dotsb p_{a_{n}}$ for a word $\a=(a_{1},\dotsc, a_{n})$. 
By carefully iterating the IFS, one can check that an $r$-stopping exists for all $r$ sufficiently small. 

Finally, we will use the following simple lemma. 
\begin{lem}\label{lem:replacebymonotone}
    If $\phi \colon [0,\infty) \to (0,1]$ satisfies $\lim_{\xi\to\infty}\phi(\xi)=0$, then there exists a continuous and strictly decreasing function $\psi \colon [0,\infty) \to (0,1]$ such that $\lim_{\xi\to\infty}\psi(\xi)=0$, $\phi = o(\psi)$, and $\psi(\xi) \geq (1+\xi)^{-1}$ for all $\xi\geq 0$. 
\end{lem}
\begin{proof}
    Since $\phi(\xi) \to 0$, there exists a sequence $0 < \xi_1 < \xi_2 < \dotsb$ such that $\max_{\xi\geq \xi_{n}}\{\phi(\xi),(1+\xi)^{-1}\} \leq 3^{-(n+2)}$ for all $n \in \N$. Defining the piecewise-affine function $\psi$ such that $\psi(0) = 1$ and $\psi(\xi_n) = 2^{-n}$ and such that $\psi$ is affine on $[0,\xi_1]$ and $[\xi_n,\xi_{n+1}]$ for each $n \in \N$ completes the proof.
\end{proof}

\subsection{Proof of Theorem \ref{t:pushisrajchman}}\label{sec:proofpushisrajchman} 
First, we prove Theorem~\ref{t:pushisrajchman}. We use Theorem~\ref{t:bakerbanaji} for part~\eqref{i:pushrajchman} and the Poincar\'e recurrence theorem \cite[Theorem 1.4]{Wal} for part~\eqref{i:measureofzeroset}. 
\begin{proof}[Proof of Theorem~\ref{t:pushisrajchman}]
Let $\mu$, $f$ and $Z$ be as in the statement of this theorem. We start by proving statement~\eqref{i:pushrajchman}. Consider some small $0 < r < 1$, and let $N_r$ denote the $r$-neighbourhood of $Z$.
    Let $\mathcal{I}$ be an $r$-stopping for $\mu$. 
    Let \[ \mathcal{T} \coloneqq \{ \a \in \mathcal{I} : T_{\a}([0,1]) \cap N_r = \varnothing\}. \] 
    Then  
    \begin{align}\label{e:pushisrajchmanmain}
    \begin{split}
        |\widehat{f \mu} (\xi)|\leq \sum_{\a \in \mathcal{I}} p_{\a} |\widehat{(f \circ T_{\a}) \mu} (\xi)|&=\sum_{\a \notin \mathcal{T}} p_{\a} |\widehat{(f \circ T_{\a}) \mu} (\xi)|+\sum_{\a \in \mathcal{T}} p_{\a} |\widehat{(f \circ T_{\a}) \mu} (\xi)|\\
        &\leq \mu(N_{2r})+\sum_{\a \in \mathcal{T}} p_{\a} |\widehat{(f \circ T_{\a}) \mu} (\xi)|.
    \end{split}
    \end{align}
    Applying the chain rule, it follows from the definition of $\mathcal{T}$ that for all $\a\in \mathcal{T}$ and $x \in [0,1]$ we have 
    \[ (f\circ T_{\a})''(x)=f''(T_{\a}(x))r_{\a}^{2} \neq 0.\] 
    Thus we can apply Theorem~\ref{t:bakerbanaji} to deduce that $(f\circ T_{\a})\mu$ is Rajchman for each $\a\in \mathcal{T}$, and conclude from~\eqref{e:pushisrajchmanmain} that 
    \begin{equation}\label{e:pushrajchmanboundbynbd}
    \limsup_{|\xi|\to\infty}|\widehat{f \mu} (\xi)|\leq \mu(N_{2r}). 
    \end{equation}
    But $Z$ a closed set, so by continuity of $\mu$ from above and our assumption that $\mu(Z)=0$, 
    \begin{equation}\label{e:pushrajchmannbdzero}
    \lim_{r \to 0^+} \mu(N_{2r}) = \mu(Z) = 0. 
    \end{equation}
    Thus, since $r$ was arbitrary, \eqref{e:pushrajchmanboundbynbd} and \eqref{e:pushrajchmannbdzero} give 
    \[ \limsup_{|\xi|\to\infty}|\widehat{f \mu} (\xi)| = 0, \] 
    in other words $f \mu$ is Rajchman, proving~\eqref{i:pushrajchman}. 

    We now focus on statement~\eqref{i:measureofzeroset} of this theorem. Assume that $\mu(Z)>0$ and $f'(x) \neq 0$ for all $x \in [0,1]$. 
    We start by stating a formula for the second derivative of $S_{\a}$ for any finite word $\a$. The following holds for any $\a\in \cup_{n=1}^{\infty}\A^{n}$ and $x\in [0,1]$ by an application of the chain rule and the formula $S_{a}=f\circ T_{a}\circ f^{-1}$ (which holds for each $a\in \A$): 
    \begin{equation}\label{e:secondderivconj}
        S_{\a}''(f(x)) = (f\circ T_{\a}\circ f^{-1})''(f(x)) = \frac{r_{\a}}{(f'(x))^2}\left( r_{\a} f''(T_{\a}(x)) - \frac{f'(T_{\a}(x))}{f'(x)} \cdot f''(x) \right).
    \end{equation}
    Here we have used our assumption that $f'(x) \neq 0$ for all $x \in [0,1]$. 
    It follows from~\eqref{e:secondderivconj} that if $f''(T_{\a}(x))=0$ and $f''(x)=0$ then $S_{\a}''(f(x))=0$. Thus to prove our result, it suffices to show that there exists a finite word $\a$ such that for a set of $x \in (0,1)$ of positive $\mu$ measure we have $f''(T_{\a}(x))=0$ and $f''(x)=0$. This we do below. 

    Let $(p_{a})_{a\in \A}$ be the probability vector so that $\mu$ is the self-similar measure corresponding to $\{T_{a}\}$ and $(p_{a})$. Define 
    \[ \pi \colon \A^{\N}\to \R, \qquad \pi((a_i)) \coloneqq \lim_{N\to\infty}(T_{a_{1}}\circ \cdots \circ T_{a_{N}})(0) \] 
    and let $m$ be the Bernoulli measure on $\A^{\N}$ corresponding to $(p_{a})_{a\in \A}$. It is a well known fact that $\mu=\pi m$. 
    Thus our assumption $\mu(Z)>0$ is equivalent to $m(\tilde{A})>0$ where 
    \[ \tilde{A} \coloneqq \{(a_i) \in \mathcal{A}^{\mathbb{N}} : f''(\pi((a_i))) = 0 \}. \] 
    Define the usual left shift map by 
    \[ \sigma \colon \A^{\N} \to \A^{\N}, \qquad \sigma((a_1,a_2,\dotsc)) = (a_2,a_3,\dotsc). \]
    It is well known that $\sigma$ is a measure preserving transformation with respect to $(\A^{\N},\mathcal{B},m)$ where $\mathcal{B}$ is the usual $\sigma$-algebra generated by the cylinder sets. Thus by Poincar\'{e}'s recurrence theorem \cite[Theorem~1.4]{Wal}, since $m(\tilde{A})>0$ there must exist $n\in \N$ such that $m(\tilde{A}\cap \sigma^{-n}(\tilde{A}))>0$. This in particular implies that there exists $\a=(c_{1},\dotsc,c_{n})$ such that 
    \[ m\left(\left\{(a_i)\in \tilde{A}: \exists (b_i)\in \tilde{A} \textrm{ such that } (a_i)=(c_{1},\dotsc,c_{n},b_{1},\dotsc)\right\}\right)>0. \] 
    This in turn implies that if we let 
    \[ \tilde{B} \coloneqq \{(b_{i})\in \tilde{A}:(c_{1},\dotsc, c_{n},b_{1},b_{2}, \dotsc )\in \tilde{A}\} \] 
    then $m(\tilde{B})>0$ and consequently $\mu(\pi(\tilde{B}))>0$. Let $x\in \pi(\tilde{B})$ and $(b_{i})\in \tilde{B}$ be such that $\pi((b_{i}))=x$. Then $f''(x)=f''(\pi((b_{i})))=0$ because $(b_{i})\in \tilde{A}$. Moreover, 
    \[ f''(T_{\a}(x))=f''(T_{\a}(\pi((b_{i})))=f''(\pi(c_{1},\dotsc, c_{n},b_{1},\dotsc ))=0 \] 
    because $(c_{1},\dotsc, c_{n},b_{1},\dotsc )\in \tilde{A}$. 
    Here we have used that 
    \[ T_{\a}(\pi((b_i)))=\pi((c_{1},\dotsc, c_{n},b_{1},\dotsc )) \qquad \mbox{ for all } (b_i)\in \A^{\N} \] 
    by the definition of the map $\pi$. In summary, for every $x\in \pi(\tilde{B})$ we have $f''(x)=0$ and $f''(T_{\a}(x))=0$, hence $S_{\a}''(f(x))=0$. 
    Since $\mu(\pi(\tilde{B}))>0$, this completes our proof of statement~\eqref{i:measureofzeroset}. 
\end{proof}

\subsection{Proof of Theorem \ref{t:slowzerodim}}\label{ss:provezerodim}

Next, we give a detailed proof of Theorem~\ref{t:slowzerodim}. We do this before proving Theorem~\ref{t:slowconformal} because the proof of Theorem~\ref{t:slowconformal} contains many of the same ideas but has the added complication of showing that for any composition of our contractions the zero set of the second derivative is finite. 
For the rest of this subsection we fix a function $\phi$ satisfying the assumptions of Theorem~\ref{t:slowzerodim} and let $\psi$ be the function from Lemma~\ref{lem:replacebymonotone} corresponding to $\phi$.

We first introduce a general framework for constructing our pushforward functions that will be used in the proofs of Theorems~\ref{t:slowconformal} and~\ref{t:slowzerodim}. To each frequency $\xi > 0$ we associate a number 
    \begin{equation}\label{e:relaterandxi}
    r = r(\xi) = (\psi(\xi))^{(\log \log (1/\psi(\xi)))^{-1}}. 
    \end{equation}
    Note that $r(\xi) \to 0$ as $\xi \to \infty$, and by the strict monotonicity of $\psi$, there exists $\xi^*>0$ such that $r\colon [\xi^*,\infty)\to (0,r(\xi^*)]$ is a bijection. Define $\xi^* < \xi_1 < \xi_2 < \dotsb$ (different to those in the proof of Lemma~\ref{lem:replacebymonotone}) such that $\xi_n \to \infty$ as $n \to \infty$, and growing slowly enough that
    \begin{equation}\label{e:unifctsfine}
    \frac{\xi_{n+1} - \xi_n}{\psi(\xi_{n+1})} \to 0 \qquad \mbox{ as } \qquad n \to \infty,
    \end{equation}
    and such that the corresponding values $r_1 > r_2 > \dotsb$ satisfy \begin{equation}\label{e:rclosespaced}
    r_{n+1} \geq r_{n} - e^{-1/r_n} \qquad \mbox{ for all } n. 
    \end{equation} Both $(\xi_{n})$ and $(r_{n})$ will be fixed throughout this subsection and the next. 
    
     Now, define $W \colon \RR \to \RR$ by 
     \begin{equation}\label{e:basicsmoothfn} 
    W(x) \coloneqq 
    \begin{cases} 
    e^{-1/x^2} \qquad &\mbox{ for } x > 0, \\* 
    0 \qquad &
    \mbox{ for } x \leq 0. 
    \end{cases}
    \end{equation}
    It is a standard fact that $W$ is $C^{\infty}$ but not real analytic, and all of the derivatives of $W$ vanish at $0$. Given a sequence $(c_{n})_{n=1}^{\infty}$ satisfying $c_{n}\in (0,2^{-n}]$ for all $n$, for $x \geq 0$ define 
    \begin{equation}\label{e:defineg}
    g(x) = \sum_{n=1}^{\infty} c_n W(x-r_{n+1}). 
    \end{equation}
    The function $g$ clearly depends upon $(c_{n})$ and $(r_{n})$, but we will suppress this dependence from our notation. We will eventually require $(c_{n})$ to satisfy more than just $c_{n}\in (0,2^{-n}]$ for all $n$. For $k = 0,1,2,\dotsc$ (with the convention $W^{(0)} = W$), using the fact $c_{n}\in (0,2^{-n}]$ for all $n$, we see that for all $R > 0$ we have
    \[ \sup_{x \in [0,R]} \sum_{n=N}^{\infty} c_n W^{(k)}(x-r_{n+1}) \leq \sum_{n=N}^{\infty} 2^{-n} \max_{x \in [0,R]} W^{(k)}(x) \leq 2^{-(N-1)} \max_{x \in [0,R]} W^{(k)}(x) \xrightarrow[N \to \infty]{} 0 . \] 
    We have shown that the series $\sum_{n=1}^{\infty} c_n W^{(k)}(x - r_{n+1})$ converges uniformly over $x \in [0,R]$. 
    Extending $g$ to $\R$ by letting $g(x) = -g(-x)$ for all $x < 0$, we see that $g$ becomes a well-defined $C^{\infty}$ function on $\R$. 
    In fact, we can interchange the order of differentiation and summation as many times as we like, so $g$ and all its derivatives of all orders vanish at $0$. We will define below a $C^{\infty}$ function $f$ whose second derivative is $g$. 
    The key property in this construction is that the sequence $(g(r_n))_{n=1}^{\infty}$ can be made to converge to zero as fast as we like. 
    In particular, it can converge to zero in a way that depends upon $\psi$, and also in a way such that $g(r_{n+1})$ is substantially less than $g(r_{n})$ for all $n\geq 1$ (see Lemma~\ref{lem:fixc_nandy_n}). 
    This property will be a consequence of~\eqref{e:defineg} and the flexibility we have in our choice of $(c_n)$. 

Now define $f \colon \R\to \R$ according to the rule
    \begin{equation}\label{e:definef} 
    f(x) = 
    \begin{cases} 
    \int_0^x \int_0^y g(r) dr dy \qquad \mbox{ for } x \geq 0, \\* 
    \int_{x}^{0} \int_{y}^{0} g(r) dr dy \qquad \mbox{ for } x<0. 
    \end{cases}
    \end{equation} 
    Since $g(x) = -g(-x)$ for all $x\in \R$, it follows that $f(x) = -f(-x)$ for all $x\in \R$. Moreover, by the fundamental theorem of calculus, $f$ is a $C^{\infty}$ function with $f''(x) = g(x)$ for all $x\in \R$. Note that $f$ and all of its derivatives vanish at $0$. Since $g(x)> 0$ for $x> 0$, it is clear from the definition that $f$ is strictly increasing when restricted to $[0,\infty)$. 
    Using the fact that $f(x) = -f(-x)$ for all $x\in \R$, it is clear that $f$ is also strictly increasing on $\R$. 
    We record here a property of $f$ that holds for any choice of $(c_{n})$: since $g$ is increasing, for all $x \in (0,1)$ we have 
    \begin{equation}\label{e:doubleintatmostfunc} 
    f(x) =\int_0^x \int_0^y g(r) dr dy \leq g(x) \cdot \frac{x^2}{2} < g(x). 
    \end{equation}
    
    The following lemma will eventually be used to fix our sequence $(c_{n})$, which in turn will fix the function $f$ we will use in our proof of Theorem~\ref{t:slowzerodim}. 
\begin{lem}\label{lem:fixc_nandy_n}
    Let $(c_{n})$ and $(y_{n})$ be two choices of sequences of positive real numbers defined inductively as follows. Let $y_1 = 0.01 \times \min\{1, \xi_2^{-1} \}$. 
    Now assume that we have defined $y_1,\dotsc,y_n$ and $c_1,\dotsc,c_{n-1}$ for some $n \geq 1$. 
    We then let $c_{n} \in (0, 2^{-n}]$ be small enough that 
    \begin{equation}\label{e:definec1}
    c_{n} W(r_j - r_{n+1}) \leq 2^{-(n+1-j)} y_j \qquad \mbox{ for all } j \in \{1,\dotsc,n\}. 
    \end{equation}
    Then let
    \begin{equation}\label{e:definey1}
    0 < y_{n+1} \leq 0.01 \times \min\{ \xi_{n+2}^{-1}, (c_n W(r_n-r_{n+1}))^{1/r_n} \} 
    \end{equation}
    be small enough that 
    \begin{equation}\label{e:ydoubleintegralsmall1}
    \int_0^{r_{n+1}} \int_0^y y_{n+1} dr dy \leq 0.01 \int_{r_{n+1}}^{r_n} \int_{r_{n+1}}^{y} c_n W(r-r_{n+1}) dr dy. 
    \end{equation}
    Then the following properties hold:
    \begin{enumerate}
        \item\label{i:boundbyy} For all $n\geq 1$ we have $g(r_{n})\leq y_{n}$.
        \item\label{i:gboundneighbouring} For all $n\geq 1$ we have $g(r_{n+1})\leq 0.01(g(r_{n}))^{1/r_{n}}$.
        \item\label{i:fboundneighbouring} For all $n\geq 1$ we have $f(r_{n+1})\leq 0.01 f(r_{n})$.
        \item\label{i:boundbyxi} For all $n\geq 1$ we have $\xi_{n+1}f(r_{n})\leq 0.01$.
    \end{enumerate}
\end{lem}
\begin{proof}
Statement~\eqref{i:boundbyy} is a consequence of the fact that for all $n\geq 1$, by~\eqref{e:defineg} and~\eqref{e:definec1} we have
\[g(r_n) = \sum_{k=1}^{\infty} c_{n+k-1} W(r_n - r_{n+k}) \leq \sum_{k=1}^{\infty} 2^{-k} y_n = y_n.\] 

Using Statement~\eqref{i:boundbyy}, \eqref{e:defineg} and~\eqref{e:definey1}, for all $n\geq 1$ we have \[g(r_{n+1}) \leq y_{n+1} \leq 0.01 (c_n W(r_n-r_{n+1}))^{1/r_n} \leq 0.01 (g(r_n))^{1/r_n},\] 
so Statement~\eqref{i:gboundneighbouring} holds. 

Using that $g$ is increasing together with~\eqref{e:defineg}, \eqref{e:ydoubleintegralsmall1}, and Statement~\eqref{i:boundbyy} implies the following for all $n\geq 1$:
\begin{align*}
    f(r_{n+1})= \int_0^{r_{n+1}} \int_0^y g(r) dr dy&\leq \int_0^{r_{n+1}} \int_0^y y_{n+1} dr dy\nonumber\\
    &\leq 0.01 \int_{r_{n+1}}^{r_n} \int_{r_{n+1}}^{y} c_n W(r-r_{n+1}) dr dy\nonumber\\
    &\leq 0.01 \int_{r_{n+1}}^{r_n} \int_{r_{n+1}}^{y} g(r) dr dy \nonumber\\
    &\leq 0.01 f(r_n). 
    \end{align*}
Thus Statement~\eqref{i:fboundneighbouring} holds. 

Statement~\eqref{i:boundbyxi} is a consequence of the following inequalities which follow from~\eqref{e:doubleintatmostfunc}, \eqref{e:definey1} and Statement~\eqref{i:boundbyy}:
\[ \xi_{n+1} f(r_{n}) \leq \xi_{n+1} g(r_{n}) \leq \xi_{n+1} y_{n} \leq 0.01. \qedhere \]
\end{proof}

Equipped with this lemma we can now prove Theorem~\ref{t:slowzerodim} using Theorem~\ref{t:bakerbanaji} and Theorem~\ref{t:pushisrajchman}~\eqref{i:pushrajchman}. 
\begin{proof}[Proof of Theorem~\ref{t:slowzerodim}]
    Let $(c_{n})$ and $(y_n)$ be defined as in Lemma~\ref{lem:fixc_nandy_n}, and let $g$ and $f$ be the functions corresponding to $(c_{n})$ defined by~\eqref{e:defineg} and~\eqref{e:definef}. 
    Let $\mu$ be any self-similar measure whose support contains $0$ and let $s_M$ be the constant from Proposition~\ref{prop:frostman} associated to $\mu$. Then for all $r\in (0,1)$, 
    \begin{equation}\label{e:bitnearzero} 
    f\mu((-f(r),f(r)) = \mu((-r,r)) \gtrsim r^{s_M}. 
    \end{equation}
    Moreover, 
    \begin{equation}\label{e:middlebit}
    \mu((r,r+2r^{4s_M/s_F})) + \mu((-r-2r^{4s_M/s_F},-r)) \lesssim r^{4s_M}.
    \end{equation} 
       For $r\in(0,1)$ let $\mathcal{I}_{r}$ be a $r^{4s_M/s_F}$-stopping for $\mu$. Let 
    \begin{equation}\label{e:defineoutercutset} 
    \mathcal{S}_r \coloneqq \{ \a \in \mathcal{I}_{r} : S_{\a}([0,1]) \cap (-r-r^{4s_M/s_F},r+r^{4s_M/s_F}) = \varnothing \}. 
    \end{equation}
    By Theorem~\ref{t:bakerbanaji} and the equality $(f\circ S_{\a})'' = r_{\a}^{2} (g\circ S_{\a})$, which follows by the chain rule and the fact $f''=g$, there is $\eta > 0$ depending only on $\mu$ and $f$ such that for all $r>0$, $\a \in \mathcal{S}_r$, $\xi \in \R$, 
    \begin{equation}\label{e:outerbit}
    |\widehat{(f\circ S_{\a})\mu}(\xi)| \lesssim \left(\min_{|x| \geq r+r^{4s_M/s_F}}|g(x)| \right)^{-1} r^{-8s_M/s_F} |\xi|^{-\eta}. 
    \end{equation}

    Now let $n \geq 1$ and let $k_n$ denote the largest positive integer such that 
    \begin{equation}\label{e:0.01bound}
    \xi_{k_n} f(r_n) \leq 0.01. 
    \end{equation}
    It follows from Statement~\eqref{i:boundbyxi} from Lemma~\ref{lem:fixc_nandy_n} that $k_{n-1} \geq n$
    for all $n$ sufficiently large. 
    If $n$ is sufficiently large then $\xi_{k_n + 1} < \xi_{k_n} + 1 < 2\xi_{k_n}$, so Statement~\eqref{i:fboundneighbouring} from Lemma~\ref{lem:fixc_nandy_n} implies $k_{n+1} > k_n$ for $n$ sufficiently large. 
    Let $m \in \N$ be such that $k_{n-1} < m \leq k_n$. 
    Then by~\eqref{e:0.01bound} and the fact that $x \mapsto e^{2 \pi i x}$ is Lipschitz with Lipschitz constant $2 \pi$, 
    \begin{align*}
    \biggl|\int_{-{r_n}}^{r_n} e^{2\pi i \xi_m f(x)} d\mu(x) - \mu((-r_n,r_n)) \biggr| &= \biggl| \int_{-{r_n}}^{r_n} (e^{2\pi i \xi_m f(x)} - 1) d\mu(x) \biggr| \\
    &\leq \mu((-r_n,r_n)) \sup_{-r_n \leq x \leq r_n} |e^{2\pi i \xi_m f(x)} - 1|  \\
    &\leq 0.01 \times 2\pi \mu((-r_n,r_n)). 
    \end{align*}
 Rearranging gives 
    \begin{equation}\label{e:easypropertyoff}
    \biggl|\int_{-{r_n}}^{r_n} e^{2\pi i \xi_m f(x)} d\mu(x) \biggr| \geq \frac{1}{2} \mu((-r_n,r_n)). 
    \end{equation} 
    By~\eqref{e:rclosespaced}, we know that 
    \begin{equation}
        \label{e:rspacing2} 
        r_{n-2} - r_n \leq 2e^{-1/r_{n-2}}\leq \frac{r_{n-2}^{4s_M/s_F}}{2} \leq r_{n}^{4s_M/s_F}
    \end{equation}
    for all $n$ sufficiently large. In the final inequality of~\eqref{e:rspacing2} we used that $\lim_{n\to\infty}\frac{r_{n+1}}{r_{n}}=1$, which follows from~\eqref{e:rclosespaced}. 
    Therefore, for all $n$ sufficiently large, for all $|x| \geq r_n + r_n^{4s_M/s_F}$, by Statement~\eqref{i:gboundneighbouring} from Lemma~\ref{lem:fixc_nandy_n}, \eqref{e:doubleintatmostfunc}, \eqref{e:rspacing2} and since $m > k_{n-1}$, we have
    \begin{equation}\label{e:secondderivboundpushforward}
        |g(x)| \geq g(r_{n-2}) \geq (g(r_{n-1}))^{r_{n-2}} \geq (f(r_{n-1}))^{r_{n-2}} \gtrsim \xi_m^{-r_{n-2}} = \xi_m^{-o_n(1)}. 
    \end{equation}
   Moreover, using that $m > k_{n-1} \geq n$,~\eqref{e:relaterandxi} and that $\psi(\xi) \geq (1+\xi)^{-1}$ we have 
    \begin{equation}\label{e:rboundpushforward}
    r_n^{-8s_M/s_F} = (\psi(\xi_n))^{-8s_M (\log \log (1/\psi(\xi_n)))^{-1}/s_F} \leq (\psi(\xi_m))^{-8s_M (\log \log (1/\psi(\xi_n)))^{-1}/s_F} \lesssim \xi_m^{o_n(1)}.
    \end{equation}
    Plugging~\eqref{e:secondderivboundpushforward} and~\eqref{e:rboundpushforward} into~\eqref{e:outerbit}, and then using $\psi(\xi) \geq (1+\xi)^{-1}$ and~\eqref{e:relaterandxi}, gives
    \begin{equation}\label{e:farawaypieces}
    \max_{\a \in \mathcal{S}_{r_n}} |\widehat{(f\circ S_{\a})\mu}(\xi_m)| \lesssim \xi_m^{-\eta/2} \lesssim (\psi(\xi_m))^{\eta/2} \lesssim r_m^{4s_M} \lesssim r_n^{4s_M}.
    \end{equation}
    
    Now, using the definition of the Fourier transform and piecing together~\eqref{e:bitnearzero}, \eqref{e:middlebit}, \eqref{e:farawaypieces} gives that 
    \begin{align*}
    |\widehat{f\mu}(\xi_m)| &\gtrsim \biggl|\int_{-{r_n}}^{r_n} e^{2\pi i \xi_m f(x)} d\mu(x)\biggr| \\*
    &\phantom{----} - \mu((r_n,r_n+2r_n^{4s_M/s_F}) \cup (-r_n-2r_n^{4s_M/s_F},-r_n)) \\*
    &\phantom{----} - \sum_{\a \in \mathcal{S}_{r_n}} p_{\a} |\widehat{(f\circ S_{\a})\mu}(\xi_m)| \\
    &\gtrsim \mu((-r_n,r_n)) /2 - r_n^{4s_M} - \max_{\a \in \mathcal{S}_{r_n}} |\widehat{(f\circ S_{\a})\mu}(\xi_m)| &\text{by~\eqref{e:easypropertyoff}, \eqref{e:middlebit}} \\ 
    &\gtrsim r_n^{s_M} - r_n^{4s_M} -  r_n^{4s_M}  &\text{by~\eqref{e:bitnearzero}, \eqref{e:farawaypieces}} \\ 
    &\gtrsim r_n^{s_M} \gtrsim \psi(\xi_n) \gtrsim \psi(\xi_m),
    \end{align*}
    where the last line holds for all $n$ sufficiently large using~\eqref{e:relaterandxi}. 
    
    In summary we have shown that $|\widehat{f\mu}(\xi_m)|  \gtrsim \psi(\xi_m)$ for all $m$ sufficiently large. Since $f \mu$ is a compactly supported Borel probability measure, its Fourier transform is Lipschitz continuous, so by~\eqref{e:unifctsfine}, the inequality $|\widehat{f\mu}(\xi_m)|  \gtrsim \psi(\xi_m)$ holding for all $m$ sufficiently large implies that in fact $|\widehat{f\mu}(\xi)| \gtrsim \psi(\xi)$ for all $\xi$ sufficiently large. 
    Since $\phi = o(\psi)$, for all $\xi$ sufficiently large we have $|\widehat{f\mu}(\xi)| \geq \phi(\xi)$. 
    If $\xi$ is negative and sufficiently large then a similar proof gives that $|\widehat{f\mu}(\xi)| \geq \phi(-\xi)$, so we have proved~\eqref{e:zerodimliminf}. 
    By Theorem~\ref{t:pushisrajchman}~\eqref{i:pushrajchman}, $f\mu$ is Rajchman. 
    This completes the proof. 
\end{proof}

\subsection{Proof of Theorem \ref{t:slowconformal}}\label{ss:proof-slow-conformal}
The proof of Theorem~\ref{t:slowconformal} has many similar features to the proof of Theorem~\ref{t:slowzerodim} but also some differences. To avoid repetition between our proofs we will on occasion not give quite as many details as in the proof of Theorem~\ref{t:slowzerodim}. 
Throughout this subsection we again fix a function $\phi$ satisfying the assumptions of Theorem~\ref{t:slowconformal} and let $\psi$ be the function from Lemma~\ref{lem:replacebymonotone} corresponding to $\phi$. 
Now and for the proof of Proposition~\ref{p:slowpushforward} below, we let the sequences $(\xi_{n})$ and $(r_{n})$ be as in the previous subsection using the relation~\eqref{e:relaterandxi} involving $\psi$, so that~\eqref{e:unifctsfine} and~\eqref{e:rclosespaced} are satisfied.
The function $g\colon \R\to \R$ will be defined for $x \geq 0$ as in~\eqref{e:defineg} where $(c_{n})_{n=1}^{\infty}$ is some sequence satisfying $c_{n}\in (0,2^{-n}]$ for all $n$. 
For $x<0$ we let $g(x) = -g(-x)$. The definition of $f \colon \R\to \R$, however, is slightly different: we now set 
    \begin{equation}\label{e:seconddefinef} 
    f(x) = 
    \begin{cases} 
    x + \int_0^x \int_0^y g(r) dr dy \qquad \mbox{ for } x \geq 0, \\* 
    x + \int_{x}^{0} \int_{y}^{0} g(r) dr dy \qquad \mbox{ for } x<0. 
    \end{cases}
    \end{equation}
    In this context we have the following analogue of~\eqref{e:doubleintatmostfunc} which follows by the same argument: for $x\in[0,1]$ we have 
    \begin{equation}
        \label{e:f-xatmostg}
        f(x)-x\leq g(x).
    \end{equation}
    
   The following proposition provides sufficient conditions for the sequence $(c_{n})$ to ensure that the Rajchman property holds and we have our desired slow Fourier decay. 
   Crucially though, there is a degree of flexibility in how we choose this sequence; this flexibility will be exploited when we prove Theorem~\ref{t:slowconformal}.
\begin{prop}
    \label{p:slowpushforward}
    Let $(c_{n})$ and $(y_{n})$ be two choices of sequences of positive real numbers defined inductively as follows. Let $y_1 = 0.01 \times \min\{1, \xi_2^{-1} \}$. 
    Assume that we have defined $y_1,\dotsc,y_n$ and $c_1,\dotsc,c_{n-1}$ for some $n \geq 1$. 
    We then let $c_{n} \in (0, 2^{-n}]$ be small enough that 
    \begin{equation}\label{e:definec2}
    c_{n} W(r_j - r_{n+1}) \leq 2^{-(n+1-j)} y_j \qquad \mbox{ for all } j \in \{1,\dotsc,n\}. 
    \end{equation}
    Then let $y_{n+1}$ be sufficiently small that 
    \begin{equation}\label{e:definey2}
    0 < y_{n+1} \leq 0.01 \times \min\{ \xi_{n+1}^{-1}, (c_n W(r_n-r_{n+1}))^{1/r_n}, r_{n+1}^{1/r_{n+1}} \}. 
    \end{equation}
     For such a choice of $(c_{n})$ and $(y_{n})$ the following properties hold:
    \begin{enumerate}
       \item\label{i:conditionsonf} $f$ is $C^{\infty}$ and strictly increasing with $f(0) = f''(0) = 0$, $f'(0) = 1$, $f'(x) \in [1,2]$ for all $x \in [0,1]$, and $f''(x) > 0$ for all $x \in (0,1]$. 
    \item\label{i:measbaseb} For every $b \geq 3$ and every proper subset $D \subset \{0,1,\dotsc, b-1 \}$ containing $0$ and at least one other digit, if $\mu$ is the base-$b$ missing digit measure corresponding to the IFS $\{x \mapsto \frac{x+d}{b}\}_{d\in D}$ and the uniform probability vector, then $f \mu$ is Rajchman, but there is some strictly increasing sequence of positive integers $(k_{n})_{n\in \N}$ and some $j\in \Z$ such that 
\[ |\widehat{f \mu}(jb^{k_n})| \geq \phi(jb^{k_n}) \quad \mbox{ for all } n \in \N. \]
\end{enumerate}
\end{prop}

\begin{proof}
    Using that $g$ is increasing, $0\leq W(x)\leq 1$ for all $x\in[0,1]$, and $c_{n}\in (0,2^{-n}]$ for all $n$, we see that for all $x \in [0,1]$ we have 
    \[ 1 \leq f'(x) = 1 + \int_0^x g(r) dr \leq 1 + g(x) = 1+\sum_{n=1}^{\infty} c_n W(x-r_{n+1})\leq 1+\sum_{n=1}^{\infty} c_n \leq 2. \] 
    We also have $f''(x) = g(x)$ for all $x$ by definition.  We see now that our $f$ satisfies the requirements of~\eqref{i:conditionsonf}. We remark that using~\eqref{e:definec2}, \eqref{e:definey2} and duplicating the arguments used in the proof of Lemma~\ref{lem:fixc_nandy_n}, it is straightforward to show that the following properties hold:
    \begin{enumerate}
        \item[(A)] For all $n\geq 1$ we have $g(r_{n})\leq y_{n}$. 
        \item[(B)] For all $n\geq 1$ we have $g(r_{n+1})\leq 0.01(g(r_{n}))^{1/r_{n}}$.
    \end{enumerate}

    Now, let $\mu$ be a base-$b$ missing digit measure as in~\eqref{i:measbaseb}, and let $s_{\mu} \coloneqq \frac{\log \# \mathcal{D}}{\log b}$. The measure $\mu$ has the additional property that $s_{\mu}=s_{F}=s_{M}$ where $s_{F}$ and $s_{M}$ are as in Proposition~\ref{prop:frostman}. 
    Since $D \neq \{0,1,\dotsc, b-1 \}$, $\mu$ is not the Lebesgue measure restricted to $[0,1]$, so there must exists $j\in \mathbb{Z}$ such that $\widehat{\mu}(j)\in \C\setminus\{0\}$. 
    Using that $\mu = T_b \mu$ where $T_b(x) = bx \mod 1$, we have 
    \[ \widehat{\mu}(j) = \widehat{\mu}(jb) = \widehat{\mu}(jb^2) = \dotsb. \]
    We denote this common value by $c_{\mu}$.  
    For each positive integer $n$, let $k_n$ denote the largest positive integer such that 
    \begin{equation}\label{e:makecloseatpower}
    j\cdot b^{k_n}(f(b^{-n}) - b^{-n}) \leq 0.01 |c_{\mu}| .
    \end{equation}
    By Statement~(A) and~\eqref{e:definey2} we have $g(r_n) \leq y_n \leq r_n^{1/r_n}$. Using the fact that $\lim_{n\to\infty}\frac{r_{n+1}}{r_{n}}=1$ and $g(r_n)\leq r_n^{1/r_n}$, it can be shown that for all $m \geq 1$ and all $x>0$ sufficiently small, we have $g(x) \leq x^{m+1}$. Thus by~\eqref{e:f-xatmostg}, for all $m\in \N$ we have 
    \[ 0 < \frac{f(b^{-n}) - b^{-n}}{b^{-mn}} \leq \frac{g(b^{-n})}{b^{-nm}} \to 0. \]
    This means that $n/k_n \to 0$ as $n \to \infty$. 
    In particular, for all sufficiently large $n$ we have $k_n > n$. 
    Using this inequality together with the fact $x \mapsto e^{2\pi ix}$ is $\mathbb{Z}$-periodic, we see that for $n$ sufficiently large the random variables $x \mapsto \chi_{[0,b^{-n}]}(x)$ and $x \mapsto e^{-2\pi i xjb^{k_{n}}}$ are independent with respect to $\mu$. Thus
    \begin{equation}\label{e:smallpieceatpower}
    \int_{0}^{b^{-n}} e^{-2 \pi i xj b^{k_n}} d \mu(x) =\int_{0}^{1}\chi_{[0,b^{-n}]}(x) e^{-2 \pi i x jb^{k_n}} d \mu(x)= b^{-s_{\mu} n} c_{\mu} 
    \end{equation} 
    for all $n$ sufficiently large. In the above we have used that $\mu([0,b^{-n}]) = b^{-s_{\mu}n}$, which is the case since $\mu$ is a missing digit measure whose alphabet contains the digit $0$. 
    Therefore 
    \begin{align*}
    \Big| \int_{0}^{b^{- n}} e^{2 \pi i f(x) jb^{k_n}} d \mu(x) - b^{-s_{\mu} n} c_{\mu}  \Big| &\leq \int_{0}^{b^{-n}} |e^{2 \pi i f(x)j b^{k_n}} - e^{2 \pi i x jb^{k_n}} | d\mu(x) &\text{by~\eqref{e:smallpieceatpower}} \\ 
    &\leq b^{-s_{\mu} n} \sup_{x \in [0,b^{-n}]} |e^{2 \pi i f(x)j b^{k_n}} - e^{2 \pi i xj b^{k_n}} | \\
    &\leq b^{-s_{\mu}n}2\pi jb^{k_{n}}\max_{x\in [0,b^{-n}]}|f(x)-x|\\
    &\leq 2\pi \times 0.01 |c_{\mu}| b^{-s_{\mu} n}  &\text{by~\eqref{e:makecloseatpower}}.
    \end{align*}
   In the final line we used that the function $f(x)-x$ restricted to $[0,b^{-n}]$ is maximised at $b^{-n}$, which is a consequence of~\eqref{e:seconddefinef}. The above implies that
    \begin{equation}\label{e:conformalbitnearzero} 
    \Big| \int_{0}^{b^{-n}} e^{2 \pi i f(x) jb^{k_n}} d \mu(x) \Big| \geq 0.9 b^{-s_{\mu} n} |c_{\mu}| 
    \end{equation}for all $n$ sufficiently large.

    For $n\geq 1$ let 
    \begin{equation}\label{e:defineoutercutset2} 
    \mathcal{S}_{b^{-n}} \coloneqq \{ \a \in {\mathcal{D}}^{4n} : S_{\a}([0,1]) \cap (-b^{-n}-b^{-4n},b^{-n}+b^{-4n}) = \varnothing \}. 
    \end{equation} 
    It is a consequence of~\eqref{e:rclosespaced} that for $n$ sufficiently large there exists $l_{n}\in \N$ such that 
    \[ b^{-n}\leq r_{l_{n}+1}<r_{l_{n}}\leq b^{-n}+b^{-4n}.\] 
    Using these inequalities together with the fact $g$ is increasing, Statement~(B), \eqref{e:f-xatmostg} and the definition of $k_{n}$, we see that the following inequalities hold for all $n$ sufficiently large: 
    \begin{equation}\label{e:gbound}
    \min_{|x|\geq b^{-n}+b^{-4n}}|g(x)|\geq g(r_{l_{n}})\geq (g(r_{l_{n}+1}))^{r_{l_{n}}}\geq (g(b^{-n}))^{{r_{l_{n}}}}\geq (f(b^{-n})-b^{-n})^{r_{l_{n}}}\gtrsim b^{-k_{n}\cdot o_{n}(1)}.
    \end{equation}
    A similar application of Theorem~\ref{t:bakerbanaji} as in~\eqref{e:outerbit} and~\eqref{e:farawaypieces}, this time also using~\eqref{e:gbound}, yields that for some $\eta > 0$ depending only on $\mu$ and $f$ the following holds for all $n$ sufficiently large: 
    \begin{align}\label{e:conformalapplybb}
    \begin{split}
    \sum_{\a \in \mathcal{S}_{b^{-n}}} p_{\a} |\widehat{(f\circ S_{\a})\mu}(jb^{k_n})| &\leq \max_{\a \in \mathcal{S}_{b^{-n}}} |\widehat{(f\circ S_{\a})\mu}(jb^{k_n})| \\
    &\lesssim \left(\min_{|x| \geq b^{-n} + b^{-4n}}|g(x)| \right)^{-1} b^{8 n} b^{-\eta k_n} \\
    &\lesssim b^{-\eta k_n / 2} \\
    &\leq b^{-4 s_{\mu} n}
    \end{split}
    \end{align}
    (the final two inequalities hold because $n = o(k_n)$). 
    
    Next, we aim for an upper bound on $\psi(jb^{k_n})$ in terms of $n$. 
    Let $j_n$ be the largest integer such that $\xi_{j_n} \leq jb^{k_n}$ (note that $\xi_{j_n} > j b^{k_n - 1}$ for $n$ large enough), so 
    \[ \xi_{j_n} (f(b^{-n}) - b^{-n}) > b^{-2} j b^{k_n + 1} (f(b^{-n}) - b^{-n}) > 0.01 b^{-2} |c_{\mu}| \] 
    for all $n$ sufficiently large, where we used~\eqref{e:makecloseatpower} for the final inequality. 
    By~\eqref{e:f-xatmostg} this implies $\xi_{j_n} g(b^{-n}) > 0.01 b^{-2} |c_{\mu}|$ for $n$ large enough. 
    By~\eqref{e:rclosespaced}, for $n$ sufficiently large there exists $l_{n}^{*}\in \N$ such that 
    \[ b^{-n}\leq r_{l_{n}^{*}+1}\leq r_{l_{n}^*}\leq b^{-(n-1)}. \]
    Using these inequalities together with Property~(B) and the fact that $g$ is increasing, we have
    \begin{align}\label{e:biggerthan1}
    \begin{split}
    \xi_{j_n} g(b^{-(n-1)})\geq \xi_{j_n} g(r_{l_{n}^*})\geq \xi_{j_n} (g(r_{l_{n}^{*}+1}))^{r_{l_{n}^*}} &\geq \xi_{j_n} (g(b^{-n}))^{r_{l_{n}^*}} \\
    &\geq (g(b^{-n}))^{r_{l_{n}^*}-1}0.01 b^{-2} |c_{\mu}| \\
    &> 1
    \end{split}
    \end{align}
    for $n$ sufficiently large. 
    In the final inequality we used that $\lim_{n\to\infty}(g(b^{-n}))^{r_{l_{n}^*}-1} = \infty$. 
    By Property~(A) and~\eqref{e:definey2} we have $\xi_{j_n}g(r_{j_n}) \leq \xi_{j_n} y_{j_n} < 1$ for all $n$. 
    Therefore by~\eqref{e:biggerthan1} and since $g$ is increasing we have $r_{j_n} < b^{-(n-1)}$ for $n$ sufficiently large. 
    Using this inequality together with~\eqref{e:relaterandxi}, $\xi_{j_{n}}\leq jb^{k_{n}}$ and the monotonicity of $\psi$ we have 
    \[ b^{-(n-1)} > r_{j_n} = (\psi(\xi_{j_n}))^{(\log \log (1/\psi(\xi_{j_n})))^{-1}} \geq (\psi(jb^{k_n}))^{1/(2s_{\mu})} \] for all $n$ sufficiently large. Rearranging, we see that $\psi(jb^{k_n}) < b^{-2 s_{\mu} (n-1)}$ for $n$ sufficiently large. This in turn implies that when $n$ is large enough we have
    \begin{equation}\label{e:psibbound}
        \psi(jb^{k_n}) < b^{- s_{\mu} n}.
    \end{equation}

    Now, using~\eqref{e:conformalbitnearzero} and~\eqref{e:conformalapplybb}, a similar computation to the main calculation from the proof of Theorem~\ref{t:slowzerodim} gives the following for $n$ sufficiently large: 
    \begin{align*}
        |\widehat{f\mu}(jb^{k_n})| &\gtrsim \Big| \int_{0}^{b^{-n}} e^{2 \pi i f(x) jb^{k_n}} d \mu(x) \Big| - \mu((b^{-n},b^{-n}+2b^{-4 n})) \\* 
        &\phantom{\gtrsim} - \sum_{\a \in \mathcal{S}_{b^{-n}}} p_{\a} |\widehat{(f\circ S_{\a})\mu}(jb^{k_n})| \\
        &\gtrsim b^{-s_{\mu} n} - b^{-4s_{\mu} n} - b^{-4 s_{\mu} n} \\
        &\gtrsim b^{-s_{\mu} n} \\
        &> \psi(jb^{k_n}),
    \end{align*}
    where the last line was~\eqref{e:psibbound}. 
    We have shown that for all $n \in \N$ sufficiently large, $|\widehat{f\mu}(jb^{k_n})| \gtrsim \psi(jb^{k_n})$, so for all $n$ large enough, $|\widehat{f\mu}(jb^{k_n})| \geq \phi(jb^{k_n})$. 
    Of course, by Theorem~\ref{t:pushisrajchman}~\eqref{i:pushrajchman}, $f\mu$ is Rajchman, completing the proof. 
\end{proof}

Next, we prove Theorem~\ref{t:slowconformal} using Proposition~\ref{p:slowpushforward} and Lemma~\ref{lem:pushforwardconformal}. 
\begin{proof}[Proof of Theorem~\ref{t:slowconformal}]
Let $\phi\colon [0,\infty)\to (0,1]$ be a function satisfying $\lim_{\xi\to\infty}\phi(\xi) = 0$. We let $\tilde{\phi}\colon [0,\infty)\to (0,1]$ be another function also satisfying $\lim_{\xi\to\infty}\tilde{\phi}(\xi) = 0$ and also the property that for all $c>0$ there exists $\xi_c > 0$ such that for all $\xi \geq \xi_c$ we have $\tilde{\phi}(\xi)\geq \phi(c\xi)$. For example, we can take $\tilde{\phi}(\xi) = \psi(\sqrt{\xi})$, where $\psi$ is the function from Lemma~\ref{lem:replacebymonotone} corresponding to $\phi$. 
Throughout the proof of Theorem~\ref{t:slowconformal}, we use the relation $r = r(\xi) = (\psi(\xi))^{(\log \log (1/\tilde{\phi}(\xi)))^{-1}}$ in place of~\eqref{e:relaterandxi}, and fix the sequences $(\xi_{n})$ and $(r_{n})$ accordingly, so $\lim_{n \to \infty} \frac{\xi_{n+1} - \xi_n}{\tilde{\phi}(\xi_{n+1})} = 0$ holds in place of~\eqref{e:unifctsfine}, and~\eqref{e:rclosespaced} holds. 

Let us fix a missing digit measure $\mu$ corresponding to the IFS 
\[ \{\lambda_{b,t}(x) \coloneqq b^{-1}(x+t)\}_{t\in \mathcal{D}} \] 
and the uniform probability vector. 
Here $\mathcal{D}$ is a proper subset of $\{0,\dotsc,b-1\}$ containing $0$ and at least one other digit. 
Let $\nu=f\mu$ be any one of the pushforward measures from Proposition~\ref{p:slowpushforward} corresponding to $\tilde{\phi}$ for $f$ corresponding to some choice of $(c_{n})$ and $(y_{n})$ satisfying the assumptions of this proposition. We see from Proposition~\ref{p:slowpushforward} and Lemma~\ref{lem:pushforwardconformal} that $\nu$ is a Rajchman self-conformal measure for some $C^{\infty}$ IFS and  
\begin{equation}
\label{e:tildeslowgrowth}
    \limsup_{\xi\to\infty}\frac{|\widehat{\nu}(\xi)|}{\tilde{\phi}(\xi)} > 0.
\end{equation}
Moreover, inspecting the proof of Lemma~\ref{lem:pushforwardconformal} and using the fact that $1\leq f'(x)\leq 2$ for all $x\in [0,1]$, it is clear that we can take this IFS to be 
\begin{equation}\label{e:definest} 
\{ S_t \coloneqq f \circ \lambda_{b,t} \circ f^{-1} \colon [0,f(1)]\to [0,f(1)] \}_{t \in \mathcal{D}} 
\end{equation}
and the probability vector to be the uniform vector. 
It is a consequence of~\eqref{e:definest} that every iterate of this IFS is of the form 
\begin{equation}\label{e:definesiterate} 
\{ S_{t} \coloneqq f \circ \lambda_{b^{l},t} \circ f^{-1} \colon [0,f(1)]\to [0,f(1)] \}_{t \in \mathcal{D}_{l}} 
\end{equation}
for some $l\in \mathbb{N}$ and $\mathcal{D}_{l}\subset \{0,1,\dotsc, b^{l}-1\}$ (here $\lambda_{b^{l},t}(x) \coloneqq b^{-l}(x+t)$).\footnote{Note that the maps $S_{t}$ for $t\in \mathcal{D}_{l}$ implicitly depend upon $l$; we suppress this dependence from our notation.} 
It is a consequence of $0\in \mathcal{D}$ that $0\in \mathcal{D}_{l}$ for all $l\in \N$.

We now focus on the following claim: 

\noindent \textbf{Claim~1:} There exist $(c_{n})$ and $(y_{n})$ satisfying the assumptions of Proposition~\ref{p:slowpushforward} such that the corresponding $f$ has the property that for every $l\in \N$ and $t\in \mathcal{D}_{l}$, the function $S_t''$ vanishes at at most finitely many points in $[0,f(1)]$. 

Despite this statement perhaps seeming obvious, the proof of this claim is somewhat technical. 
To verify this claim, we fix a choice of $(y_{n})$ and introduce a parameterised family of $(c_{n})$ for which each member of this family and $(y_{n})$ satisfies the assumptions of Proposition~\ref{p:slowpushforward} with the sequences $(\xi_{n})$ and $(r_{n})$ defined as described above in terms of $\tilde{\phi}$. 
For now we fix sequences $(c_{n})$ and $(y_{n})$, where $(c_{n})$ is defined by at each stage setting 
\[ c_n \coloneqq \frac{1}{2} \min\Big\{2^{-n}, \min_{1 \leq j \leq n} \frac{2^{-(n+1-j)} y_j}{W(r_j - r_{n+1})} \Big\}, \]
and $(y_{n})$ is chosen arbitrarily so the assumptions of Proposition~\ref{p:slowpushforward} are satisfied. We let $g_{0}$ be defined by~\eqref{e:defineg} for this specific choice of $(c_{n})$.
For each $n\in \N$ consider the interval 
\[ \Omega_{n} \coloneqq \left[0,\frac{1}{2} \min\Big\{2^{-n}, \min_{1 \leq j \leq n} \frac{2^{-(n+1-j)} y_j}{W(r_j - r_{n+1})} \Big\}\right] \] 
equipped with the Borel $\sigma$-algebra, and let $\nu_{n}$ be the normalised Lebesgue measure on this interval. Thus $\nu_{n}$ is a probability measure. 
We now consider the set $\Omega \coloneqq \prod_{n=1}^{\infty}\Omega_{n}$ equipped with the corresponding product $\sigma$-algebra and the product measure $\nu \coloneqq \prod_{n=1}^{\infty}\nu_{n}$. 
Then $\nu$ is also a probability measure. To each $\omega = (\delta_{n})_{n=1}^{\infty}\in \Omega$ we associate the sequence $(c_{n}(\omega))_{n=1}^{\infty}$ according to the rule 
\[ c_{n}(\omega) = \frac{1}{2} \min\Big\{2^{-n}, \min_{1 \leq j \leq n} \frac{2^{-(n+1-j)} y_j}{W(r_j - r_{n+1})} \Big\}+\delta_{n} \] 
for all $n\in \N$. 
Note that for each $\omega$ the sequence $(c_{n}(\omega))$  satisfies $c_{n}(\omega)\in (0,2^{-n}]$ for all $n\in \N$ and~\eqref{e:definec2}. Since $c_{n}\leq c_{n}(\omega)$ for all $n$ and $\omega$, we still have 
\[ y_{n+1} \leq (c_n(\omega) W(r_n - r_{n+1}))^{1/r_n} \qquad \mbox{ for all } n \in \N,\] 
so~\eqref{e:definey2} still holds. This means that for every $\omega\in \Omega$ the sequences $(c_n(\omega))$ and $(y_n)$ satisfy the assumptions of Proposition~\ref{p:slowpushforward}. For each $\omega$ we let $g_{\omega}\colon \mathbb{R}\to \mathbb{R}$ be defined as in Subsection~\ref{ss:provezerodim}, i.e. for $x \geq 0$ we have 
\begin{equation}\label{e:definegomega}
g_{\omega}(x) = \sum_{n=1}^{\infty}c_{n}(\omega) W(x-r_{n}) 
\end{equation}
and $g(x) = -g(-x)$ for $x<0$. 
We similarly let 
$f_{\omega}$ and $S_{\omega,t}$ be the corresponding functions obtained from $g_{\omega}$ as in~\eqref{e:seconddefinef}, \eqref{e:definest}. 
The pushforward measures $\nu_{\omega} \coloneqq f_{\omega} \mu$ satisfy~\eqref{e:tildeslowgrowth} by Proposition~\ref{p:slowpushforward}. 
We note that for all $\omega\in \Omega$ and $x\geq 0$ we have 
\begin{equation}
    \label{e:doublebound}
    g_{0}(x)\leq g_{\omega}(x)\leq 2g_{0}(x).
\end{equation} 
Claim~$1$ follows immediately from the following claim. 

\noindent \textbf{Claim~1a:} For $\nu$ almost every $\omega\in \Omega$, for all $l\in \N$ and $t\in \mathcal{D}_{l}$ the function $S_{\omega,t}''$ vanishes at at most finitely many points in $[0,f_{\omega}(1)]$.

Moreover, as a countable union of sets with $\nu$-measure zero has $\nu$-measure zero, it is sufficient to prove the following claim.

\noindent \textbf{Claim~1b:} For each $l\in \N$, for $\nu$ almost every $\omega\in \Omega$, for all $t\in \mathcal{D}_{l}$ the function $S_{\omega,t}''$ vanishes at at most finitely many points in $[0,f_{\omega}(1)]$.

We now give a proof of Claim~1b.

\textbf{Proof of Claim~1b:} 
Let us begin by fixing $l\in \N$. A direct calculation shows that the following holds for all $\omega\in \Omega$, $x\geq 0$ and $t\in\mathcal{D}_{l}$ as in~\eqref{e:secondderivconj}: 
\begin{align}\label{e:firstandsecondderivs}
\begin{split}
&S_{\omega,t}'(f_{\omega}(x)) = b^{-l}\frac{f_{\omega}'(\lambda_{b^{l},t}(x))}{f_{\omega}'(x)}; \\ 
&S_{\omega,t}''(f_{\omega}(x)) = \frac{b^{-l}}{(f_{\omega}'(x))^2}\left( b^{-l} g_{\omega}(\lambda_{b^{l},t}(x)) - \frac{f_{\omega}'(\lambda_{b^{l},t}(x))}{f_{\omega}'(x)} \cdot g_{\omega}(x) \right).
\end{split}
\end{align}

Thus for all $\omega\in \Omega$, $t\in \mathcal{D}_{l}$ and $x > 0$, the sign of $S_{\omega,t}''(f_{\omega}(x))$ is the same as the sign of 
\begin{equation}
\label{e:samesign*}
    b^{-l} \frac{g_{\omega}(\lambda_{b^{l},t}(x))}{g_{\omega}(x)} - \frac{f_{\omega}'(\lambda_{b^{l},t}(x))}{f_{\omega}'(x)}. 
\end{equation} 
But for all $\omega\in \Omega$ and $x \in [0,1]$ we have $f_{\omega}'(x) \in [1,2]$\, so 
\begin{equation}
\label{e:fquotientbounds*}
 \frac{1}{2} \leq \frac{f_{\omega}'(\lambda_{b^{l},t}(x))}{f_{\omega}'(x)} \leq 2
 \end{equation}
for $x \in [0,1]$. 
Moreover, for all $\omega\in \Omega$ and $t \in \mathcal{D} \setminus \{0\}$, 
\[ \frac{g_{\omega}(\lambda_{b^{l},t}(x))}{g_{\omega}(x)} \to \infty \quad \mbox{ and } \quad \frac{g_{\omega}(\lambda_{b^{l},0}(x))}{g_{\omega}(x)} \to 0 \quad \mbox{ as } x \to 0^+. \]
Thus, considering our unperturbed function $g_{0}$ (which corresponds to the case where $\omega=(\delta_{n})$ satisfies $\delta_{n} = 0$ for all $n$) we see that there exists some small $x_0 > 0$ such that for all $x \in (0,x_{0}]$ we have 
\begin{equation}
    \label{e:ginequalities1*}
\frac{g_{0}(\lambda_{b^{l},0}(x))}{g_{0}(x)}\leq \frac{1}{10}\quad \textrm{ and } \quad     \frac{g_{0}(\lambda_{b^{l},t}(x))}{g_{0}(x)}\geq 10b^{l}\quad  \forall \, t\in \mathcal{D}_{l} \setminus \{0\}.
\end{equation} 
It follows now from~\eqref{e:doublebound} and~\eqref{e:ginequalities1*} that the following holds for all $x \in (0,x_{0}]$ and all $\omega\in \Omega$:
\begin{equation*}
    \label{e:ginequalities2*}
\frac{g_{\omega}(\lambda_{b^{l},0}(x))}{g_{\omega}(x)}\leq \frac{1}{5}\quad \textrm{ and } \quad   \frac{g_{\omega}(\lambda_{b^{l},t}(x))}{g_{\omega}(x)}\geq 5b^{l}\quad  \forall \, t\in \mathcal{D} \setminus \{0\}.
\end{equation*}
By~\eqref{e:samesign*} and~\eqref{e:fquotientbounds*} this implies that for each $\omega\in \Omega$ we have
\begin{align}\label{e:secondderiv2*}
\begin{split}
    &S_{\omega,0}''(f(x)) < 0 \quad \mbox{ for all } x \in (0,x_0] \qquad \mbox{ (while $S_{0}''(0) = 0$)}, \mbox{ and } \\*
    &S_{\omega,t}''(f(x)) > 0 \quad \mbox{ for all } x \in [0,x_0],\, t \in \mathcal{D} \setminus \{0\}.
\end{split}
\end{align}

For each $\omega\in \Omega$, $f_{\omega}$ is analytic on each interval of the form $(r_{n+1},r_{n})$, so each $S_{\omega,t}$ is piecewise-analytic on $(0,f_{\omega}(1))$. 
Indeed, $S_{\omega,0}$ is analytic on $(0,f_{\omega}(1))$ away from the points of the form $f_{\omega}(r_n)$ or $f_{\omega}(\lambda_{b^{l},0}^{-1}(r_n))$ for some $n$. 
Moreover, if $t \in \mathcal{D}_{l}\setminus\{0\}$ then $S_{\omega,t}$ is analytic on $(0,f_{\omega}(1))$ away from the points of the form $f_{\omega}(r_{n})$ for some $n$ or $f_{\omega}(\lambda_{b^{l},t}^{-1}(r_{n}))$ for one of the finitely many $r_{n}\in \lambda_{b^{l},t}([0,1])$ (recall that $\lim_{n\to\infty}r_{n} = 0$ and $0\notin\lambda_{b^{l},t}([0,1])$ for $t\in \mathcal{D}_{l}\setminus\{0\}$). 
This set of points where analyticity fails forms a discrete set that accumulates at $0$, so the following set is finite:
\begin{equation}\label{e:defineasequence} 
(x_{0},1]\cap \left(\{r_{n}\}_{n\in \N}\cup \{\lambda_{b^{l},0}^{-1}(r_{n})\}_{n\in \N}\cup \{\lambda_{b^{l},t}^{-1}(r_{n}) : t \in \mathcal{D}_{l}\setminus\{0\}, r_{n}\in \lambda_{b^{l},t}([0,1])\}\cup\{1\}\right). 
\end{equation}
We let $\{a_{i}\}_{i=1}^{K}$ denote an enumeration of this set written in increasing order, and let $a_{0} = x_{0}$. We emphasise that $a_{K}=1$ and that $\cup_{i=0}^{K}(a_{i},a_{i+1})$ equals $(x_{0},1)$ apart from finitely many points where successive intervals meet. Moreover, by construction, for all $\omega\in \Omega$, $0\leq i\leq K-1$ and $t\in \mathcal{D}_{l}$, the function $S_{\omega,t}''$ is analytic on $(f_{\omega}(a_i),f_{\omega}(a_{i+1}))$. 
By~\eqref{e:secondderiv2*} we know that for all $\omega\in \Omega$ and $t\in \mathcal{D}_{l}$, we have $S_{\omega}''(y)\neq 0$ for any $y \in (0,f(x_{0})]$. 
Thus, to prove Claim~1b, it is sufficient to show that for $\nu$ almost every $\omega\in \Omega$, for each $t\in \mathcal{D}_{l}$ the function $S_{\omega,t}''$ has finitely many zeros in $(f_{\omega}(x_{0}),f_{\omega}(1))$. The fact that each $S_{\omega,t}''$ is analytic on $(f_{\omega}(a_{i}),f_{\omega}(a_{i+1}))$ for each $0\leq i\leq K-1$ and $(f_{\omega}(x_{0}),f_{\omega}(1))$ equals $\cup_{i=0}^{K-1}(f_{\omega}(a_{i}),f_{\omega}(a_{i+1}))$ modulo a finite set of points will be very useful in achieving this goal. 
First, however, we have to check that $\nu$ almost surely, each $S_{\omega,t}''$ is not the constant zero function on any of the intervals $(f_{\omega}(a_{i}),f_{\omega}(a_{i+1}))$. 
This is the content of the following claim. 

\textbf{Claim~2:} For $\nu$ almost every $\omega\in \Omega$, for all $t\in \mathcal{D}_{l}$ and $0\leq i\leq K-1$, there exists $y\in (f_{\omega}(a_{i}),f_{\omega}(a_{i+1}))$ such that $S_{\omega,t}''(y)\neq 0$.

We start with an observation. Suppose that $\omega\in \Omega$, $t\in \mathcal{D}_{l}$ and $0\leq i\leq K-1$ were such that $S_{\omega,t}''(y) = 0$ for all $y\in (f_{\omega}(a_{i}),f_{\omega}(a_{i+1}))$. Then $S_{w,t}'$ would be a constant function on $(f_{\omega}(a_{i}),f_{\omega}(a_{i+1}))$. 
Combining this observation with~\eqref{e:firstandsecondderivs} implies that if $\omega\in \Omega$, $t\in \mathcal{D}_{l}$ and $0\leq i\leq K-1$ were such that $S_{\omega,t}''(y)= 0$ for all $y\in (f_{\omega}(a_{i}),f_{\omega}(a_{i+1}))$, then there exists $C_{\omega}>0$ such that 
\[ g_{\omega}(\lambda_{b^{l},t}(x))=C_{\omega}g_{\omega}(x) \qquad \mbox{ for all } x\in (a_{i},a_{i+1}). \] 

Now let us approach Claim~2 directly. Suppose the claim is false. Then there exists $t\in \mathcal{D}_{l}$ and $0\leq i\leq K-1$ such that for some $B\subset \Omega$ with $\nu(B)>0$, for all $\omega\in B$ we have $S_{\omega,t}''(y) = 0$ for all $y\in (f_{\omega}(a_{i}),f_{\omega}(a_{i+1}))$. By the above, this in turn implies that there exists $t\in \mathcal{D}_{l}$ and $0\leq i\leq K-1$, such that for all $\omega\in B$ there exists $C_{\omega}>0$ such that 
\begin{equation}
\label{e:gconstants}
g_{\omega}(\lambda_{b^{l},t}(x)) = C_{\omega} g_{\omega}(x) \qquad \mbox{ for all } x\in (a_{i},a_{i+1}).
\end{equation}
Recalling the definition of $g_{\omega}$ from~\eqref{e:definegomega}, the definition of $\{a_0,\dotsc,a_K\}$ from~\eqref{e:defineasequence}, and the fact that $x_0 = a_0 \leq a_i$, we see that no points in $(a_i,a_{i+1})$ coincide with any of the $r_n$ values, so there exists a positive integer $n_1$ such that 
\begin{equation}\label{e:equalsparticularsumbasic2}
g_{\omega}(x) = \sum_{n = n_1}^{\infty} c_n(\omega) e^{-(x-r_{n+1})^{-2}} \qquad \mbox{ for all } x \in (a_i,a_{i+1}). 
\end{equation}
Similarly, for all $x \in (a_i,a_{i+1})$, $\lambda_{b^{l},t}(x)$ does not coincide with any of the $r_n$ values, so there exists a positive integer $n_2$ such that 
\begin{equation}\label{e:equalsparticularsumscaled2} 
g_{\omega}(\lambda_{b^{l},t}(x)) = \sum_{n = n_2}^{\infty} c_n(\omega) e^{-(\lambda_{b^{l},t}(x)-r_{n+1})^{-2}} \qquad \mbox{ for all } x \in (a_i,a_{i+1}). 
\end{equation}
Thus by~\eqref{e:gconstants}, for all $x \in (a_i,a_{i+1})$ we have 
\begin{equation}
    \label{e:constantseries}
    \sum_{n=n_{2}}^{\infty} c_{n}(\omega)e^{-1/(\lambda_{b^{l},t}(x)-r_{n+1})^{2}} = C_{\omega}\sum_{n=n_{1}}^{\infty}c_{n}(\omega) e^{-1/(x-r_{n+1})^{2}}.
\end{equation}
 
 Let us now fix $M\geq \max\{n_{1},n_{2}\}$. 
 Since $\nu$ is a product measure and $\nu(B)>0$, it is a consequence of Fubini's theorem applied to the indicator function on $B$ that there exists some sequence $(\delta_{n}^{*})_{n=1,n\neq M}^{\infty}$ such that the set 
 \[ B_{M} \coloneqq \left\{\delta: (\delta_{1}^{*},\dotsc, \delta_{M-1}^{*},\delta,\delta_{M+1}^{*},\dotsc )\in B\right\} \] 
 satisfies $\nu_{M}(B_{M})>0$. 
 For $\delta\in B_{M}$ we let $(c_{n}(\delta))_{n=1}^{\infty}$ be the sequence defined by 
 \begin{equation*}
    c_{n}(\delta) \coloneqq 
    \begin{cases} 
    \frac{1}{2} \min\Big\{2^{-n}, \min_{1 \leq j \leq n} \frac{2^{-(n+1-j)} y_j}{W(r_j - r_{n+1})} \Big\}+\delta_{n}^* \qquad &\mbox{ for } n \neq M, \\* 
    \frac{1}{2} \min\Big\{2^{-M}, \min_{1 \leq j \leq M} \frac{2^{-(M+1-j)} y_j}{S(r_j - r_{M+1})} \Big\} + \delta \qquad &
    \mbox{ for } n = M. 
    \end{cases}
    \end{equation*}
 We emphasise that for all $n\neq M$, the term $c_{n}(\delta)$ does not depend upon $\delta$. 
 It follows from~\eqref{e:constantseries} that for each $\delta\in B_{M}$ there exists $C_{\delta}>0$ such that for all $x \in (a_{i},a_{i+1})$ we have 
 \begin{equation}
 \label{e:constantseries2}
 \sum_{n=n_{2}}^{\infty} c_{n}(\delta) e^{-1/(\lambda_{b^{l},t}(x)-r_{n+1})^{2}} = C_{\delta}\sum_{n=n_{1}}^{\infty}c_{n}(\delta) e^{-1/(x-r_{n+1})^{2}}.
 \end{equation}
 
 Let us now fix $\delta^{*}\in B_{M}$. 
 If $\delta\in B_{M}$ then by~\eqref{e:constantseries2} and the definitions of $(c_{n}(\delta))$ and $(c_{n}(\delta^*))$, for all $x \in (a_{i},a_{i+1})$ we have
\begin{equation}
\label{e:writetwoways1}
\sum_{n=n_{2}}^{\infty}c_{n}(\delta) e^{-1/(\lambda_{b^{l},t}(x)-r_{n+1})^{2}} = C_{\delta}\left(\sum_{n=n_{1}}^{\infty} c_{n}(\delta^{*}) e^{-1/(x-r_{n+1})^{2}} + (\delta-\delta^{*}) e^{-1/(x-r_{M+1})^{2}}\right).
\end{equation}
Using~\eqref{e:constantseries2} again, we see that the following property for the left hand side of~\eqref{e:writetwoways1} also holds for all $\delta\in B_{M}$ and all $x \in (a_{i},a_{i+1})$:
\begin{equation}
\label{e:writetwoways2}
\sum_{n=n_{2}}^{\infty} c_{n}(\delta) e^{-1/(\lambda_{b^{l},t}(x)-r_{n+1})^{2}} = C_{\delta^*} \sum_{n=n_{1}}^{\infty}c_{n}(\delta^{*}) e^{-1/(x-r_{n+1})^{2}} + (\delta-\delta^{*}) e^{-1/(\lambda_{b^{l},t}(x) - r_{M+1})^{2}}.
\end{equation}
Combining~\eqref{e:writetwoways1} and~\eqref{e:writetwoways2} yields that for all $\delta\in B_{M}$ and all $x\in (a_{i},a_{i+1})$ we have 
\begin{equation}
\label{e:obtaincontradiction}
(C_{\delta^{*}}-C_{\delta}) \sum_{n=n_{1}}^{\infty} c_{n}(\delta^{*}) e^{-1/(x-r_{n+1})^{2}} = C_{\delta}(\delta-\delta^*) e^{-1/(x-r_{M+1})^{2}} - (\delta-\delta^{*}) e^{-1/(\lambda_{b^{l},t}(x)-r_{M+1})^{2}}.
\end{equation}

We now show via a case analysis that~\eqref{e:obtaincontradiction} always gives a contradiction. In our analysis we will repeatedly use the following fact. 

\noindent \textbf{Statement~2:} For all $i$ and $t\in \mathcal{D}_{l}$, the function 
\[ x \mapsto \frac{e^{-1/(\lambda_{b^{l},t}(x)-r_{M+1})^{2}}}{e^{-1/(x-r_{M+1})^{2}} } \] 
is not constant on the interval $(a_{i},a_{i+1})$. 

Statement~$2$ is a simple consequence of fact that for all $t\in \mathcal{D}_{l}$, the function 
\[ x \mapsto (\lambda_{b^{l},t}(x) - r_{M+1})^{-2} - (x - r_{M+1})^{-2} \] 
is not a constant function on any $(a_{i},a_{i+1})$. 
We now give our case analysis.

\noindent \textbf{Case A: $C_{\delta} = C_{\delta^*}$ for some $\delta\in B_{M}\setminus \{\delta^{*}\}$.} 
In this case~\eqref{e:obtaincontradiction} implies 
\[ 0 = C_{\delta}(\delta-\delta^*) e^{-1/(x-r_{M+1})^{2}}- (\delta-\delta^{*}) e^{-1/(\lambda_{b^{l},t}(x)-r_{M+1})^{2}} \qquad \mbox{ for all } x\in (a_{i},a_{i+1}). \] 
This is not possible by Statement~$2$.

\noindent \textbf{Case B: $C_{\delta}\neq C_{\delta^*}$ for all $\delta\in B_{M}\setminus \{\delta^*\}$.} 
In this case we can rewrite~\eqref{e:obtaincontradiction} so that for all $\delta\in B_{M}\setminus \{\delta^*\}$ and all $x\in (a_{i},a_{i+1})$ we have 
\[ \sum_{n=n_{1}}^{\infty}c_{n}(\delta^{*}) e^{-1/(x-r_{n+1})^{2}} = \frac{C_{\delta}(\delta-\delta^*)}{(C_{\delta^{*}}-C_{\delta})} e^{-1/(x-r_{M+1})^{2}} - \frac{(\delta - \delta^{*})}{(C_{\delta^{*}}-C_{\delta})} e^{-1/(\lambda_{b^{l},t}(x)-r_{M+1})^{2}}. \]  
Since the left hand side of the above does not depend upon $\delta$, it follows that for any $\delta,\delta'\in B_{M}\setminus \{\delta^{*}\}$ and $x\in (a_{i},a_{i+1})$ we have 
\begin{align}
\begin{split}
    \label{e:bracketedterms}
    0 &= \left(\frac{C_{\delta}(\delta-\delta^*)}{(C_{\delta^{*}}-C_{\delta})}-\frac{C_{\delta'}(\delta'-\delta^*)}{(C_{\delta^{*}}-C_{\delta'})}\right) e^{-1/(x-r_{M+1})^{2}}\\*
    &\phantom{-} -\left(\frac{(\delta-\delta^{*})}{(C_{\delta^{*}}-C_{\delta})} - \frac{(\delta'-\delta^{*})}{(C_{\delta^{*}}-C_{\delta'})}\right) e^{-1/(\lambda_{b^{l},t}(x)-r_{M+1})^{2}}.
    \end{split}
    \end{align}
 It follows from Statement~$2$ that the two bracketed terms in~\eqref{e:bracketedterms} equal zero for all $\delta,\delta'\in B_{M}\setminus \{\delta^{*}\}$. 
 In summary, we have shown that the maps $\delta \mapsto \frac{C_{\delta}(\delta-\delta^*)}{(C_{\delta^{*}}-C_{\delta})}$ and $\delta \mapsto \frac{(\delta-\delta^{*})}{(C_{\delta^{*}}-C_{\delta})}$ are constant functions on $B_{M}\setminus \{\delta^*\}$. This in turn implies that the map $\delta \mapsto C_{\delta}$ is a constant function on $B_{M}\setminus \{\delta^*\}$. 
 We denote the constant value this function takes by $C$. 
 Substituting this information into~\eqref{e:obtaincontradiction} yields that for every $\delta\in B_{M}\setminus \{\delta^*\}$ and $x\in (a_{i},a_{i+1})$ we have 
\begin{equation}
\label{e:final contradiction}
(C_{\delta^*}-C) \sum_{n=n_{1}}^{\infty} c_{n}(\delta^{*}) e^{-1/(x-r_{n+1})^{2}}=(\delta-\delta^*)(Ce^{-1/(x-r_{M+1})^{2}} - e^{-1/(\lambda_{b^{l},t}(x)-r_{M+1})^{2}}).
\end{equation}
 The left hand side of~\eqref{e:final contradiction} does not depend upon $\delta$. 
 The right hand side of~\eqref{e:final contradiction} is of the form $(\delta-\delta^{*})$ multiplied by a term that does not depend upon $\delta$. 
 As $\delta$ is an arbitrary element of $B_{M}\setminus \{\delta^*\}$ and $B_{M}$ is a positive measure set (and is therefore infinite\footnote{In fact $\#B_{M}\geq 3$ is sufficient for this argument.}), the only way for~\eqref{e:final contradiction} to hold for all $x\in (a_{i},a_{i+1})$ and $\delta\in B_{M}\setminus\{\delta^{*}\}$ is for 
 \[ Ce^{-1/(x-r_{M+1})^{2}} - e^{-1/(\lambda_{b^{l},t}(x)-r_{M+1})^{2}} = 0 \qquad \mbox{ for all } x \in (a_{i},a_{i+1}). \] 
 However, this is not possible by Statement~$2$.

 We have shown that regardless of whether we are in Case~(A) or Case~(B), we obtain a contradiction to Statement~$2$. Thus, Claim~$2$ must be true. 

We return to the proof of Claim~1b. By Claim~2, the set of $\omega\in \Omega$ such that for all $t \in \mathcal{D}_{l}$ and $0 \leq i \leq K-1$ there exists $y\in (f_{\omega}(a_{i}),f_{\omega}(a_{i+1}))$ such that $S_{\omega,t}''(y)\neq 0$ has full $\nu$-measure. 
Therefore it suffices to show that for any such $\omega$, for all $t\in \mathcal{D}_{l}$ the function $S_{\omega,t}''$ vanishes at at most finitely many points in $[0,f_{\omega}(1)]$. 
For the remainder of the proof we fix such an $\omega$ and write $(c_{n})=(c_{n}(\omega))$, $g=g_{\omega}$, $f = f_{\omega}$ and $S_t = S_{\omega,t}$ for brevity. 
Our strategy will be to show that for each $t \in \mathcal{D}_{l}$, at all points $y \in (0,f(1))$ which are sufficiently close to some $f(a_i)$ (where $0 \leq i \leq K$) but do not actually coincide with any of these points, we have $S_t''(y) \neq 0$. 
This together with~\eqref{e:secondderiv2*} and the fact that $S_{t}''$ is analytic on $(f(a_{i}),f(a_{i+1}))$ for each $0\leq i\leq K-1$  will imply that $S_{t}''$ has finitely many zeros. 

Consider the interval $[f(a_i),f(a_{i+1})]$ for some $i \in \{0,\dotsc,K - 1\}$. 
If $t\in \mathcal{D}_{l}$ is such that $S_t''(f(a_{i+1})) \neq 0$ then by continuity there exists $\varepsilon_1 > 0$ such that $S_t''(y) \neq 0$ for all $y \in (f(a_{i+1}) - \varepsilon_1, f(a_{i+1}) + \varepsilon_1)$. 
We now assume that $t\in \mathcal{D}_{l}$ is such that $S_t''(f(a_{i+1})) = 0$. 
As in~\eqref{e:equalsparticularsumbasic2} and~\eqref{e:equalsparticularsumscaled2} there exist $n_{1}\in \N$ and $n_{2}\in \N$ such that 
\begin{equation}\label{e:equalsparticularsumbasic}
g(x) = \sum_{n = n_1}^{\infty} c_n e^{-(x-r_{n+1})^{-2}} \quad \mbox{ and } \quad g(\lambda_{b^{l},t}(x)) = \sum_{n = n_2}^{\infty} c_n e^{-(\lambda_{b^{l},t}(x)-r_{n+1})^{-2}}  
\end{equation}
for all $x\in (a_{i},a_{i+1}]$. Note that $r_{n_1+1} \leq a_i$ and define the strictly increasing function $\tilde{f} \colon [r_{n_1+1},\infty) \to \R$ by 
\begin{equation*}
\tilde{f}(x) \coloneqq 
\begin{cases} x + \int_0^x \int_0^y \sum_{n = n_1}^{\infty} c_n e^{-(r-r_{n+1})^{-2}} dr dy, \qquad &\mbox{ for } x \in (r_{n_1+1},\infty), \\* 
r_{n_1+1} + \int_0^{r_{n_1+1}} \int_0^y \sum_{n = n_1 + 1}^{\infty} c_n e^{-(r-r_{n+1})^{-2}} dr dy, \qquad &\mbox{ for } x = r_{n_1+1}.
\end{cases}
\end{equation*}
Note that $\tilde{f}$ is analytic on $(r_{n_1+1},\infty)$. 
We emphasise that for all $x \in [a_i,a_{i+1}]$ we have $\tilde{f}(x) = f(x)$, $\tilde{f}'(x) = f'(x)$, and $\tilde{f}''(x) = g(x)$. 
Define the function $h \colon [\tilde{f}(r_{n_1+1}),\infty) \to \R$ by 
\begin{align}
\begin{split}\label{e:longhformula}
h(y) &\coloneqq \frac{b^{-l}}{(\tilde{f}'(\tilde{f}^{-1}(y)))^2} \cdot b^{-l} \sum_{n = n_2}^{\infty} c_n e^{-(\lambda_{b^l,t}(\tilde{f}^{-1}(y)) - r_{n+1})^{-2}} \\*
&- \frac{b^{-l}}{(\tilde{f}'(\tilde{f}^{-1}(y)))^2} \cdot \frac{1 + \int_0^{\lambda_{b^l,t}(\tilde{f}^{-1}(y))} \sum_{n = n_2}^{\infty} c_n e^{-(x'-r_{n+1})^{-2}} dx'}{\tilde{f}'(\tilde{f}^{-1}(y))} \cdot \sum_{n = n_1}^{\infty} c_n e^{-(\tilde{f}^{-1}(y)-r_{n+1})^{-2}}
\end{split}
\end{align}
for $y \in (\tilde{f}(r_{n_1+1}),\infty)$, extending to $\tilde{f}(r_{n_1+1})$ by continuity. Note that $h$ is analytic on $(\tilde{f}(r_{n_1+1}),\infty)$. 
By~\eqref{e:firstandsecondderivs}, \eqref{e:equalsparticularsumbasic}, $h|_{(f(a_i),\infty)}$ is the unique analytic function on $(f(a_i),\infty)$ which satisfies 
\[ h(f(x)) = S_t''(f(x)) \qquad \mbox{ for all } x \in (a_i,a_{i+1}]. \]

Now, since $S_t'' \circ f$ is not identically zero on $(a_i,a_{i+1})$ by assumption and coincides with $h \circ \tilde{f}$ on this interval, we see that $h \circ \tilde{f}$ is not identically zero on $(a_i,\infty)$. Therefore, since $h\circ \tilde{f}$ is analytic on $(a_i,\infty)$, there exists some minimal positive integer $k$ such that $(h \circ \tilde{f})^{(k)}(a_{i+1}) \neq 0$. 
But since $f = \tilde{f}$ on $[a_i,a_{i+1}]$, we have 
\[ (h \circ f)^{(k)}(a_{i+1}) = (h \circ \tilde{f})^{(k)}(a_{i+1}) \neq 0. \]
Therefore by Taylor's theorem there exist small constants $C,\varepsilon_2 > 0$ such that  
\begin{equation}\label{e:positivekthderiv} 
|h(f(x))| \geq C (x - a_{i+1})^k \qquad \mbox{ for all } x \in (a_{i+1} - \varepsilon_2, a_{i+1} + \varepsilon_2).
\end{equation}
In particular, 
\[ h(f(x)) \neq 0 \qquad \mbox{ for all } x \in (a_{i+1} - \varepsilon_2, a_{i+1} + \varepsilon_2) \setminus \{a_{i+1}\},\] 
so 
\[ S_t''(f(x)) = h(f(x)) \neq 0 \qquad \mbox{ for all } x \in (a_{i+1} - \varepsilon_2, a_{i+1}). \] 

Understanding the behaviour of $S_t''$ to the right of $a_{i+1}$ is a little more delicate. Define the $C^{\infty}$ function $s_{0}$ via the equation $\tilde{f}=f+s_{0}$. 
Using the definitions of $g$, $f$, the formula for $S_t''$ given in~\eqref{e:firstandsecondderivs} and the formula for $h$ from~\eqref{e:longhformula}, we can find $\varepsilon_3 > 0$ and similarly-defined $C^{\infty}$ functions $s_1,s_2,s_3,s_4$ (some of which may be identically zero) such that for all $x \in [a_{i+1}, a_{i+1} + \varepsilon_3)$, 
\begin{align}\label{e:approxsecondderivflat}
\begin{split}
&h(f(x) + s_0(x)) \\*
& = \frac{b^{-l} (g(\lambda_{b^{l},t}(x)) + s_1(x)) (f'(x) + s_2(x)) - (f'(\lambda_{b^{l},t}(x)) + s_3(x)) (g(x) + s_4(x))}{b^{l}(f'(x) + s_2(x))^3}. 
\end{split}
\end{align}
Moreover, we can guarantee that for all $i \in \{0,1,2,3,4\}$, $s_i$ and all its derivatives vanish at $a_{i+1}$, so (decreasing $\varepsilon_3$ if necessary) 
\begin{equation}
\label{e:s_i faster than poly decay}
 s_i(x) < (x - a_{i+1})^{k+1} \qquad \mbox{ for all } i \mbox{ and } x \in (a_{i+1}-\varepsilon_3,a_{i+1}+\varepsilon_3). 
 \end{equation}

Recalling that we have chosen $s_0$ such that $\tilde{f} = f + s_0$, by the mean value theorem there exists $C>0$ such that 
\begin{equation}\label{e:hcloseattwoclosevalues} 
|h(\tilde{f}(x)) - h(f(x))| \leq C s_0(x) \qquad \mbox{ for all } x \in [0,1].
\end{equation}
Letting $\varepsilon_4>0$ be small enough that $g(x), g(\lambda_{b,t}(x)), f'(x), f'(\lambda_{b,t}(x))$ remain uniformly bounded away from $0$ and $\infty$ for $x \in [a_{i+1} - \varepsilon_4,a_{i+1} + \varepsilon_4]$, we see from~\eqref{e:firstandsecondderivs}, \eqref{e:approxsecondderivflat} and a direct algebraic manipulation that there is some large constant $C'>0$ such that for all such $x$ we have 
\[ |h(\tilde{f}(x)) - S_t''(f(x))| < C' \max_{1 \leq i \leq 4} s_i(x). \] 
But by~\eqref{e:hcloseattwoclosevalues} this means that there exists $C'' \geq C'$ and $\varepsilon_5 \in (0,\varepsilon_4)$ such that for all $x \in [a_{i+1} - \varepsilon_5,a_{i+1} + \varepsilon_5]$ we have 
\[ |h(f(x)) - S_t''(f(x))| < C'' \max_{0 \leq i \leq 4} s_i(x). \] 
Moreover there is some small $0 < \varepsilon_6 < \min\{ \varepsilon_1, \varepsilon_2, \varepsilon_3,\varepsilon_4,\varepsilon_5\}$ such that for all $x \in [-\varepsilon_6,\varepsilon_6]$ we have $C x^k/2 > C''x^{k+1}$. 
This means that for all $x \in (a_{i+1}, a_{i+1} + \varepsilon_6]$, by the above and \eqref{e:s_i faster than poly decay}, we have 
\[ |h(f(x)) - S_t''(f(x))| < \frac{C(x - a_{i+1})^{k}}{2}, \] 
and therefore by~\eqref{e:positivekthderiv}, 
\begin{equation*}
|S_t''(f(x))| > \frac{C(x - a_{i+1})^{k}}{2} > 0. 
\end{equation*}

Since $0 \leq i \leq K-1$ in the above argument was arbitrary, we see that there exists 
\[ 0 < \varepsilon_7 < \min \left\{ \varepsilon_6, \min_{0 \leq i \leq K-1} \frac{a_{i+1} - a_i}{2} \right\} \] such that 
\[ |S_t''(f(x))| > 0 \qquad \mbox{ for all } \qquad x \in \left( \bigcup_{i=0}^{K} [a_i - \varepsilon_7,a_i + \varepsilon_7] \setminus \{a_i\} \right) \cap (0,1) \]
(the $i=0$ case follows since $S_t''(f(a_0)) = S_t''(f(x_0)) \neq 0$ by~\eqref{e:secondderiv2*}). 
Then since $S_t''$ coincides with $h$ on $[f(a_i + \varepsilon_7), f(a_{i+1})]$ and this compact set is contained in the set $(f(a_i),\infty)$ on which $h$ is analytic, $S_t''$ must have at most finitely many zeros in $[f(a_i),f(a_{i+1})]$. 
Since $[f(a_{i}),f(a_{i+1})]$ was an arbitrary choice of interval, it now follows from this fact and~\eqref{e:secondderiv2*} that $S_{t}''$ can have at most finitely many zeros in $[0,f(1)]$. Since $t$ was arbitrary, this completes our proof of Claim~1b, which as discussed implies Claim~1a and Claim~1. 

To complete the proof of Theorem~\ref{t:slowconformal}, let $f$ be as in the statement of Claim~1, let $\nu$ be the resulting self-conformal measure, and let $\tilde{\nu}$ be the pushforward of $\nu$ by $\tilde{\lambda}(x) \coloneqq x/f(1)$. Then $\tilde{\nu}$ is the self-conformal measure for the IFS 
\[ \{\mathcal{S}_{t}= \tilde{\lambda} \circ f \circ \lambda_{b,t} \circ f^{-1} \circ \tilde{\lambda}^{-1} \colon [0,1]\to [0,1] \}_{t \in \mathcal{D}}. \] 
For any composition of the $\mathcal{S}_{t}$ contractions, it is a consequence of the definition of these maps and the chain rule that the zero set of the second derivative of this composition is precisely the image of the zero set of $(f \circ \lambda_{b^{l},t} \circ f^{-1})''$ under $\tilde{\lambda}$ for some $l\in \N$ and $t\in \mathcal{D}_l$, so in particular is finite. 
Now, by scaling properties of the Fourier transform, for all $\xi \in \RR$ we have $\widehat{\tilde{\nu}}(\xi) = \widehat{\nu}(\xi/f(1))$. 
Therefore since $\tilde{\phi}(\xi/f(1)) \geq \phi(\xi)$ for all $\xi$ sufficiently large by the definition of $\tilde{\phi}$, and using~\eqref{e:tildeslowgrowth}, we have 
\[ \limsup_{\xi \to \infty} \frac{|\widehat{\tilde{\nu}}(\xi)|}{\phi(\xi)} \geq \limsup_{\xi \to \infty} \frac{|\widehat{\nu}(\xi/f(1))|}{\tilde{\phi}(\xi/f(1))} > 0. \]
This completes our proof of Theorem~\ref{t:slowconformal}. 
\end{proof}

\subsection{Explicit examples}\label{ss:Examples}
In this section we will give the arguments for our remaining theorems and justify the claims made in Example~\ref{Example:self-similar image}. 
Theorems~\ref{t:particularfunctionconf} and~\ref{t:particularfunctionzero} can be proved using a similar strategy to that used in the proofs of Proposition~\ref{p:slowpushforward} and Theorem~\ref{t:slowconformal} (for Theorem~\ref{t:particularfunctionconf}), and Theorem~\ref{t:slowzerodim} (for Theorem~\ref{t:particularfunctionzero}). 
For Theorem~\ref{t:particularfunctionconf} we will need the following lemma. 

\begin{lem}\label{lem:particularzeroderiv}
    Let $h(x) = \exp(-\exp(x^{-2}))$ and $f(x) = x + h(x)$ for $x \in \R \setminus \{0\}$, and $f(0) = h(0) = 0$. 
    Let $0 < c < 1$ and $d \geq 0$, and let $\lambda_{c,d}(x) \coloneqq cx + d$ for $x \in \R$. 
    Then the following set is finite: 
    \begin{equation}\label{e:particularzeroset} 
    \{ y \in [0,f(1)] : (f \circ \lambda_{c,d} \circ f^{-1})''(y) = 0 \}.
    \end{equation}
\end{lem}
\begin{proof}
For $x \neq 0$, 
\begin{align}\label{e:derivsparticular}
\begin{split}
h'(x) &= 2x^{-3} \exp(x^{-2}) h(x) \qquad \mbox{ and } \qquad f'(x) = 1 + h'(x) ; \\* 
f''(x) = h''(x) &= \Big((2x^{-3} \exp(x^{-2}))^2 - 6 x^{-4} \exp(x^{-2}) - 4 x^{-6} \exp(x^{-2}) \Big) h(x),
\end{split}
\end{align}
while $h'(0) = h''(0) = f''(0) = 0$. 
Observe from~\eqref{e:derivsparticular} that 
\begin{equation}\label{e:particularboundedaway0infty}
1 = f'(0) = \min_{0 \leq x < \infty} f'(x) < \sup_{0 \leq x < \infty} f'(x) < \infty. 
\end{equation}
Also, differentiating once more, we see that when $x>0$ is very small, the dominant term in the expression for $f'''(x) = h'''(x)$ is $(2x^{-3} \exp(x^{-2}))^3 h(x)$. It follows that there exists $\varepsilon > 0$ such that $h'''(x) > 0$ for all $x \in (0,\varepsilon]$, so $h''$ is (strictly) increasing on $[0,\varepsilon]$. 

Furthermore, by~\eqref{e:derivsparticular}, there exist functions $P_1(x), P_2(x)$ with
\[ \max\{|P_1(x)|,|P_2(x)|\} < x^{-3} \] 
for all $x>0$ sufficiently small, such that 
\begin{align}\label{e:secondderivspecific}
\begin{split}
\frac{\log h''(x)}{\log h''(cx)} &= \frac{P_1(x) + \log h(x)}{P_2(x) + \log h(cx)} \\
&= \frac{\frac{P_1(x)}{\log h(cx)} + \frac{\log h(x)}{\log h(cx)}}{\frac{P_2(x)}{\log h(cx)} + 1} \\
&= \frac{\frac{P_1(x)}{\log h(cx)} + \exp(x^2(1-c^{-2}))}{\frac{P_2(x)}{\log h(cx)} + 1} \\
&\xrightarrow[x \to 0^+]{} \frac{0+0}{0+1} = 0. 
\end{split}
\end{align}
Thus
\begin{equation}\label{e:particularratiosecondderivsblowsup}
\lim_{x \to 0^+} \frac{h''(cx)}{h''(x)} = 0.\footnote{We only require~\eqref{e:particularratiosecondderivsblowsup} for the proof of Lemma~\ref{lem:particularzeroderiv}, but the stronger statement~\eqref{e:secondderivspecific} will be used in the proof of Theorem~\ref{t:particularfunctionconf} below.}
\end{equation}

Now, by the same calculation as in~\eqref{e:secondderivconj}, for all $x \in \R$ we have 
\begin{equation}\label{e:particularconjsecondderiv}
(f \circ \lambda_{c,d} \circ f^{-1})''(f(x)) = \frac{c}{(f'(x))^2}\left( c f''(cx+d) - \frac{f'(cx+d)}{f'(x)} \cdot f''(x) \right). 
\end{equation}
We proceed as in the derivation of~\eqref{e:secondderiv2*}, considering the two cases $d=0$ and $d>0$ separately. First consider $d=0$. 
Combining~\eqref{e:particularconjsecondderiv} with~\eqref{e:particularboundedaway0infty},~\eqref{e:particularratiosecondderivsblowsup} and noting that $f''(x) = h''(x) > 0$ for $x>0$, we see that there exists $y_0 \in (0,f(1))$ (depending on $c$) such that $(f \circ \lambda_{c,0} \circ f^{-1})''(y) < 0$ for all $y \in (0,y_0]$ (while $(f \circ \lambda_{c,0} \circ f^{-1})''(0) = 0$). 
If $d > 0$, on the other hand, then by~\eqref{e:particularboundedaway0infty} and the fact that $f''(0) = 0$ we have 
\[ \lim_{x \to 0} (f \circ \lambda_{c,d} \circ f^{-1})''(f(x)) = c^2 f''(d) > 0. \] 
Therefore, if $d > 0$ then there exists $y_0' \in (0,f(1))$ (depending on $c$ and $d$) such that $(f \circ \lambda_{c,d} \circ f^{-1})''(y) > 0$ for all $y \in [0,y_0']$. 

We have shown that (regardless of whether or not $d=0$) there exists $y_0'' \in (0,f(1))$ depending on $c,d$ such that 
\begin{equation}\label{e:particularnotzeronearzero}
(f \circ \lambda_{c,d} \circ f^{-1})''(y) \neq 0 \qquad \forall \, y \in (0,y_0''].
\end{equation}
Moreover, $(f \circ \lambda_{c,d} \circ f^{-1})''$ is analytic on $(0,\infty)$, and by~\eqref{e:particularnotzeronearzero} it is not identically zero. 
Therefore it has at most finitely many zeros in the compact interval $[y_0,f(1)]$. This together with~\eqref{e:particularnotzeronearzero} implies the finiteness of the set from~\eqref{e:particularzeroset}. 
\end{proof}

We sketch the proof of Theorem~\ref{t:particularfunctionconf}, but leave out the parts of the proof that are very similar to the proof of Proposition~\ref{p:slowpushforward}. 
\begin{proof}[Proof sketch of Theorem~\ref{t:particularfunctionconf}]
Let $h(x) = \exp(-\exp(x^{-2}))$ and $f(x) = x + h(x)$ for $x \neq 0$, and $f(0) = h(0) = 0$, and let $\mu$ be the Cantor--Lebesgue measure supported inside $[0,1]$. 
As in the proof of Theorem~\ref{t:slowconformal}, $f \mu$ is a self-conformal measure for the $C^{\infty}$ IFS $\{f \circ T_1 \circ f^{-1}, f \circ T_2 \circ f^{-1}\}$, where $T_{1}(x)=x/3$ and $T_{2}(x)=(x+2)/3$. 
By Lemma~\ref{lem:particularzeroderiv}, all iterates of the IFS have second derivative vanishing at at most finitely many points in $[0,f(1)]$. 

Now, recall the observation from Lemma~\ref{lem:particularzeroderiv} that there exists $\varepsilon>0$ such that $h''$ (and therefore also $f''$) is increasing on $[0,\varepsilon]$. Without loss of generality we can assume $\varepsilon=3^{-M}$ for some large $M\in \N$. Then for all $\xi\in \R$ we have
\begin{equation}\label{e:particularsplitintegralfirst}
\int_{\R} e^{2\pi i \xi f(x)} d\mu(x)=\int_{0}^{3^{-M}} e^{2\pi i \xi f(x)} d\mu(x)+\int_{3^{-M}}^{\infty} e^{2\pi i \xi f(x)} d\mu(x).
\end{equation}
For the latter expression we have 
\[ \int_{3^{-M}}^{\infty} e^{2\pi i \xi f(x)} d\mu(x) = \sum_{\a\in \{1,2\}^{M}:\a\neq 0^{M}}p_{\a}\int_{\R} e^{2\pi i \xi f(T_{\a}(x))} d\mu(x).\]
Recall that it was proved in~\cite{ACWW25,BB25} that analytic non-affine images of self-similar measures have polynomial Fourier decay. Hence each integral appearing in the summation above decays to zero polynomially fast in $\xi$. Thus it remains to show that the remaining integral over $[0,3^{-M}]$ has the desired slow decay. 

Given an integer $n > M$, let $k_n$ be the largest integer such that $3^{k_n} h(3^{-n}) \leq 0.01$. Since $\supp(\mu) \cap (3^{-n},2\cdot (3^{-n})) = \varnothing$, we can further divide the integral: 
\begin{equation}\label{e:particularsplitintegralsecond}
\int_{0}^{3^{-M}} e^{2\pi i 3^{k_{n}} f(x)} d\mu(x) = \int_{0}^{3^{-n}} e^{2\pi i 3^{k_{n}} f(x)} d\mu(x)  + \int_{2\cdot (3^{-n})}^{3^{-M}} e^{2\pi i 3^{k_{n}} f(x)} d\mu(x). 
\end{equation}
For all $n>M$ sufficiently large, we can deal with the $2\cdot (3^{-n}) \leq x \leq 3^{-M}$ region by obtaining an upper bound for 
\[ \sum_{\a \in \{1,2\}^{n-M}: \a \neq 0^{n-M}} p_{0^M \a} \Big| \int_{\R} e^{2\pi i 3^{k_n} f(T_{0^M \a}(x))} d\mu(x) \Big| \] 
using the same proof strategy as in Proposition~\ref{p:slowpushforward}. 
Indeed, to do this we use the $c=1/2$ case of~\eqref{e:secondderivspecific} and the fact that $h''$ is increasing on $[0,3^{-M}]$ to obtain a lower bound on $\max_{2\cdot (3^{-n}) \leq x \leq 3^{-M}} h''(x)$, and apply Theorem~\ref{t:bakerbanaji}. An analogue of~\eqref{e:conformalbitnearzero} for the integral over $[0,3^{-n}]$ can be established using the definition of $k_{n}$, properties of the Cantor--Lebesgue measure, and the fact that $\widehat{\mu}(3^{n})=\widehat{\mu}(1)\neq 0$ for all $n\in \N$. These bounds for the two parts of the integral in~\eqref{e:particularsplitintegralsecond} together yield that for all large enough $n>M$, 
\[ \int_{0}^{3^{-M}} e^{2\pi i 3^{k_n} f(x)} d\mu(x) \gtrsim \mu([0,3^{-n}]) = 2^{-n}. \]
But for all large enough $n$, $3^{k_n} \approx \exp(\exp(9^{n}))$, so $\log(3^{k_n}) \approx k_n \approx \exp(9^{n})$, hence $\log \log (3^{k_n}) \approx 9^{n}$. 
Therefore for all large enough $n>M$, 
\[ \int_{0}^{3^{-M}} e^{2\pi i 3^{k_n} f(x)} d\mu(x) \gtrsim 2^{-n} > 9^{-n} \approx \frac{1}{\log \log (3^{k_n})}. \]
Using~\eqref{e:particularsplitintegralfirst} and the aforementioned polynomial decay of $\int_{3^{-M}}^{\infty} e^{2\pi i \xi f(x)}$ in $\xi$, this implies that for all large enough $n$, 
\[ |\widehat{f \mu}(3^{k_n})| \gtrsim \frac{1}{\log \log (3^{k_n})}. \] 
Nonetheless, $f \mu$ is Rajchman by Theorem~\ref{t:pushisrajchman}~\eqref{i:pushrajchman}. 
\end{proof}

The argument used in the proof of Theorem~\ref{t:particularfunctionzero} is essentially contained in the proofs of Theorems~\ref{t:particularfunctionconf} and~\ref{t:slowzerodim}. As such we will omit this argument. 

The discussion below justifies the claims made in Example~\ref{Example:self-similar image}. Recall that this example motivates the condition appearing in Question~\ref{Q:refinedquestion} that is formulated in terms of the second derivatives of all composition of contractions coming from our IFS.

\noindent \textbf{Example~\ref{Example:self-similar image} continued}. Recall that $\nu$ is the Cantor--Lebesgue measure on $[0,1]$ (so $T_{1}(x) = x/3$ and $T_{2}(x) = (x+2)/3$ with the $(1/2,1/2)$ probability vector). In Example~\ref{Example:self-similar image} we choose $f \colon [0,1]\to [0,1]$ an increasing $C^{\infty}$ diffeomorphism which is affine on $[2/9,1/3]$ and satisfies
\begin{equation}
    f''(x) \neq 0 \qquad \mbox{ for all } x \in [0,2/9) \cup (3/9,1]. 
    \end{equation}  We claimed that it is possible to choose $f$ so that the following properties hold:
    \begin{enumerate}
        \item\label{i:exampleconformalproperty} $\mu$ is the self-conformal measure for the $C^{\infty}$ IFS 
        \[ \{S_{a} \coloneqq f\circ T_{a}\circ f^{-1} \colon [0,1]\to [0,1]\}_{a=1}^{2} \] 
        and the probability vector $(1/2,1/2)$. 
        \item\label{i:examplederivproperty} For $a\in \{1,2\}$ we have $S_{a}''(x)\neq 0$ apart from at most finitely many $x\in [0,1]$.
    \end{enumerate} It remains to justify this claim. We start by choosing $f$ in such a way that it is real analytic on $[0,2/9)\cup (1/3,1]$. If we also choose $f$ in such a way that $\max_{x \in [0,1]}|f'(x)| \in (0.9,1.1)$, then by the proof of Lemma~\ref{lem:pushforwardconformal}, $\mu$ is the self-conformal measure for the IFS $\{f \circ S_1 \circ f^{-1}, f \circ S_2 \circ f^{-1} \}$ with probability vector $(1/2,1/2)$, so property~\eqref{i:exampleconformalproperty} holds. 
    
    To justify property~\eqref{i:examplederivproperty}, it is useful to recall the formula for $S_{a}''$ provided by~\eqref{e:secondderivconj}: for $x\geq 0$ and $a\in \{1,2\}$ we have
    \begin{equation}
        \label{e:secondderivformulaagain}
        S_{a}''(f(x)) = \frac{1}{3(f'(x))^{2}}\left(\frac{f''(T_{a}(x))}{3}-\frac{f'(T_{a}(x))}{f'(x)}\cdot f''(x)\right).
    \end{equation}
    If $x\in [2/9,1/3]$ then $f''(x)=0$, but for each $a \in \{1,2\}$, $f''(T_{a}(x))\neq 0$ because $T_{a}(x)\notin [2/9,1/3]$ if $x\in [2/9,1/3]$. 
    Thus $S_{a}''(f(x))\neq 0$ for all $x\in [2/9,1/3]$ and $a\in \{1,2\}$ by~\eqref{e:secondderivformulaagain}. Similarly, if $x\in [2/3,1]$ then $f''(T_{1}(x))=0$ and $f''(x)\neq 0$, so $S_{1}''(f(x))\neq 0$ for all $x\in [2/3,1]$. 
    By continuity there exists $\varepsilon>0$ such that $S_{a}''(f(x))\neq 0$ for all $x\in (2/9-\varepsilon,1/3+\varepsilon)$ and $a\in \{1,2\}$, and $S_{1}''(f(x))\neq 0$ for all $x\in (2/3-\epsilon,1]$. It remains to consider the following cases:
    \begin{enumerate}
        \item $x\in [0,2/9-\varepsilon] $ and $a=2$.
        \item $x\in [1/3+\varepsilon,1]$ and $a=2$.
        \item $x\in [0,2/9-\varepsilon]$ and $a=1$.
        \item $x\in [1/3+\varepsilon,2/3-\varepsilon]$ and $a=1$.
    \end{enumerate} We need to show that we can choose $f$ in such a way that each of the cases above contributes at most finitely many values of $x$ such that $S_{a}''(f(x))=0$. We will focus on the case where $x\in [0,2/9-\varepsilon] $ and $a=2$; the other cases can be addressed similarly. The key observation is that the function 
    \begin{equation}
    \label{e:analytic function}
    x \mapsto \left(\frac{f''(T_{2}(x))}{3}-\frac{f'(T_{2}(x))}{f'(x)}\cdot f''(x)\right)
    \end{equation} 
    is real analytic on $[0,2/9-\varepsilon]$. Thus, by~\eqref{e:secondderivformulaagain}, to prove that there are at most finitely many $x$ in the compact interval $[0,2/9-\varepsilon]$ satisfying $S_{2}''(f(x))=0$, it is sufficient to show that we can choose $f$ so that the function described by~\eqref{e:analytic function} is not the constant zero function. 
    This can be done by adapting the arguments used in the proof of Theorem~\ref{t:slowconformal}. 
    In particular, we can introduce a random perturbation that is infinitely flat at the endpoints $2/9$ and $1/3$, and show that for almost every choice of perturbation, the function described by~\eqref{e:analytic function} is not the constant zero function on $[0,2/9-\varepsilon]$. Applying this argument to each of the cases listed above, and using the fact that a finite union of sets of measure zero is also of measure zero, eventually yields an $f$ satisfying property~\eqref{i:examplederivproperty}. 

\section*{Acknowledgements}
We thank Tuomas Orponen for a conversation which helped with the construction of the $C^{\infty}$ function $g$ in Section~\ref{ss:provezerodim}, and Boris Solomyak for helpful comments on a draft version of this paper. 
Both authors were financially supported by SB's EPSRC New Investigators Award (EP/W003880/1) at Loughborough University. AB was also supported by the Marie Skłodowska-Curie Actions postdoctoral fellowship FoDeNoF (no. 101210409) from the European Union, and by Tuomas Orponen's Research Council of Finland grant (no. 355453), at the University of Jyväskylä. 

\bibliographystyle{plain}

\end{document}